\documentclass[final]{siamltex}

\usepackage{amsmath,amssymb,ifthen}
\usepackage{graphicx,psfrag,subfigure}
\usepackage[dvips]{color}

\def\H{\widetilde{H}}

\def\R{{\mathbb R}}

\def\KK{{\mathcal K}}

\def\MM{{\mathcal M}}

\def\PP{{\mathcal P}}

\def\SS{{\mathcal S}}
\def\SSS{{\mathbb S}}
\def\TT{{\mathcal T}}
\def\NN{\mathcal{N}}

\def\II{{\mathcal I}}
\def\L{\mathbf L}
\def\H{\mathbf H}

\newcommand{\norm}[3][]{#1\|#2#1\|_{#3}}

\def\set#1#2{\big\{#1\,:\,#2\big\}}

\def\Ce{C_{\rm exch}}
\def\Ca{C_{\rm ani}}
\def\eps{\varepsilon}

\newcounter{constantsnumber}
\def\namec#1#2{%
  \ifthenelse{\equal{#1}{rho:reliable}}{C_{\rm rel}}{%
  \ifthenelse{\equal{#1}{rho:efficient}}{C_{\rm eff}}{%
  \ifthenelse{\equal{#1}{stability}}{C_{\rm stab}}{%
  \ifthenelse{\equal{#1}{equivalence1}}{C_{\rm low}}{%
  \ifthenelse{\equal{#1}{equivalence2}}{C_{\rm high}}{%
  \ifthenelse{\equal{#1}{optimal}}{C_{\rm opt}}{%
  \ifthenelse{\equal{#1}{scottzhang}}{C_{\rm sz}}{%
  \ifthenelse{\equal{#1}{dirichlet0}}{C_{\rm dir}}{%
  \ifthenelse{\equal{#1}{dirichlet}}{C_{\rm osc}}{%
  \ifthenelse{\equal{#1}{equivalence}}{C_{\rm eq}}{%
  \ifthenelse{\equal{#1}{pythagoras}}{C_{\rm pyth}}{%
  \ifthenelse{\equal{#1}{dlr}}{C_{\rm dlr}}{%
  \ifthenelse{\equal{#1}{nvb}}{C_{\rm nvb}}{%
  \ifthenelse{\equal{#1}{reduction}}{C_{\rm red}}{%
  \ifthenelse{\equal{#1}{refined}}{C_{\rm ref}}{%
  \ifthenelse{\equal{#1}{estconv}}{C_{\rm est}}{%
  \ifthenelse{\equal{#1}{cea}}{C_{\mbox{\rm\scriptsize C\'ea}}}{%
  \ifthenelse{\equal{#2}{newcounter}}{\refstepcounter{constantsnumber}\label{const#1}}{}C_{\ref{const#1}}}%
}}}}}}}}}}}}}}}}}
\def\setc#1{\namec{#1}{newcounter}}
\def\c#1{\namec{#1}{reference}}

\newcounter{contractionnumber}
\def\nameq#1#2{%
  \ifthenelse{\equal{#1}{reduction}}{q_{\rm red}}{%
  \ifthenelse{\equal{#1}{estconv}}{q_{\rm est}}{%
  \ifthenelse{\equal{#1}{cea}}{q_{\mbox{\scriptsize C\'ea}}}{%
  \ifthenelse{\equal{#2}{newcounter}}{\refstepcounter{contractionnumber}\label{contraction#1}}{}q_{\ref{contraction#1}}}%
}}}

\def\mmm{\mathbf{m}}

\def\heff{\mathbf{H}_{\text{eff}}}
\def\fff{\mathbf{f}}

\def\nnn{\mathbf{n}}
\def\zzz{\mathbf{z}}
\def\xxx{\mathbf{x}}
\def\yyy{\mathbf{y}}
\def\vvv{\mathbf{v}}
\def\uuu{\mathbf{u}}

\def\pphi{\boldsymbol{\phi}}
\def\hm{\mathbf{h}_\mathbf{m}}
\def\ssigma{\boldsymbol{\sigma}}
\def\eeps{\boldsymbol{\eps}}
\def\llambda{\boldsymbol{\lambda}}
\def\vvvv{\dot \uuu}
\def\TTTT{\mathfrak{T}}

\def\ppsi{\boldsymbol{\psi}}

\def\vphi{\boldsymbol{\varphi}}
\def\zzeta{\boldsymbol{\zeta}}

\def\wgamma{\gamma_{hk}}
\def\wvvvv{\vvvv_{hk}}
\def\operator{\boldsymbol\pi}
\def\dt{\text{d}_t}
\def\bf{\boldsymbol}
\def\Cpi{C_{\operator}}
\def\weakto{\rightharpoonup}

\newtheorem{theorem}{Theorem}
\newtheorem{proposition}[theorem]{Proposition}
\newtheorem{lemma}[theorem]{Lemma}
\newtheorem{corollary}[theorem]{Corollary}
\newtheorem{algorithm}[theorem]{Algorithm}
\newtheorem{definition}[theorem]{Definition}
\newenvironment{remark}{\medskip\noindent\textbf{Remark.}\ \it}{\qed\smallskip}

\def\subsection#1
{
 \bigskip

 \refstepcounter{subsection}
 {\noindent\bf\arabic{section}.\arabic{subsection}.~#1.~}
}

\def\norm#1{\|#1\|}

\def\qed{\ifhmode\unskip\nobreak\fi\ifmmode\ifinner\else\hskip5 pt \fi\fi
 \hbox{\hskip25 pt \hbox{\vrule width .2 pt \vbox{\hrule width 4 pt 
 height .2 pt \vskip 6.2 pt \hrule width 4 pt  height .2 pt }\unskip\vrule
 width .2 pt }\hskip 0pt }}

\title{On the Landau-Lifshitz-Gilbert equation with magnetostriction}

\author{L'.~Ba\v nas\thanks{Department of Mathematics,
Heriot-Watt University, Edinburgh,
United Kingdom, {\it E-mail address:} {\tt L.Banas@hw.ac.uk}.}
        \and M.~Page\thanks{Institute for Analysis and Scientific Computing,
       Vienna University of Technology,
       Wiedner Hauptstra\ss{}e 8-10,
       A-1040 Wien, Austria, {\it E-mail address:} {\tt Marcus.Page@tuwien.ac.at} (corresponding author), {\tt Dirk.Praetorius@tuwien.ac.at}.}
       \and D.~Praetorius$^{\dagger}$
       \and J.~Rochat\thanks{MATHICSE, \'{E}cole Polytechnique F\'{e}d\'{e}rale de Lausanne, station 8, CH-1015 Lausanne, Switzerlandk {\it E-mail address:} {\tt jonathan.rochat@epfl.ch}.}}

\begin{document}

\maketitle

\begin{abstract}
To describe and simulate dynamic micromagnetic phenomena, we
consider a coupled system of the nonlinear Landau-Lifshitz-Gilbert equation
and the conservation of momentum equation. This coupling allows to 
include magnetostrictive effects into the simulations.
Existence of weak solutions has recently been shown in~\cite{carbou}.
In our contribution, we give an alternate proof which additionally provides
an effective numerical integrator. The latter is based on lowest-order
finite elements in space and a linear-implicit Euler time-stepping.
Despite the nonlinearity, only two linear
systems have to be solved per timestep, and the integrator fully decouples
both equations.
Finally, we prove unconditional convergence---at least of a 
subsequence---towards, and hence existence of,
a weak solution of the coupled system,
as timestep size and spatial mesh-size tend to zero. Numerical experiments conclude
the work and shed new light on the existence of blow-up in micromagnetic simulations.
\end{abstract}



\pagestyle{myheadings}
\thispagestyle{plain}
\markboth{L'.~BA\v NAS, M.~PAGE, D.~PRAETORIUS, AND J.~ROCHAT}{CONVERGENT SCHEME FOR LLG WITH MAGNETOSTRICTION}

\section{Introduction}\label{sec:intro}
\noindent
Throughout all technical areas, magnetic devices like sensors, recording heads,
and magneto-resistive storage devices are quite popular and thus widely used. 
As their size decreases to a microscale, and the testing and development 
becomes more and more involved, the need for reliable and stable 
simulation tools as well as for a thorough theoretical understanding rises.
In terms of mathematical physics,
micromagnetic phenomena are modeled best by the Landau-Lifshitz-Gilbert 
equation (LLG), see~\eqref{eq:llg_total} below. This nonlinear partial 
differential equation describes the behaviour of the magnetization of some 
ferromagnetic body under the influence of a so-called effective field. 
The mathematical challenges as well as its applicability to a wide range of 
real world problems makes LLG an interesting problem for mathematicians and 
physicists, but also for scientists from related fields like engineers and 
developers from high-tech industry.

In our contribution, we present and analyze a computationally attractive integrator
to solve LLG numerically. Additionally, our analysis provides a constructive
existence proof for weak solutions of the coupled system for LLG with 
magnetostriction and thus particularly includes the results of~\cite{carbou}.
For the pure LLG equation,
existence and non-uniqueness of weak solutions of LLG goes back 
to~\cite{as, visintin} for a simplified effective field. 
For a review of the analysis of LLG, we refer 
to~\cite{cimrak, gc, mp06} or the monographs~\cite{hubertschaefer, prohl} and 
the references therein. As far as the numerical analysis is concerned, 
mathematically reliable and convergent LLG integrators are found 
in~\cite{alouges2008, alouges2011, banas, maxwell,bjp,bp,multiscale, tran, mathmod2012, gps,rochat}. Of utter interest are unconditionally convergent integrators which do not impose a coupling of spatial mesh-size $h$ and time-step size $k$ to ensure stability of the numerical integrator.
Those integrators are split into two 
groups; first, midpoint-scheme based integrators~\cite{banas,rochat} which 
rely on the seminal work~\cite{bp} of \textsc{Bartels \& Prohl}; second, projection-based first-order 
integrators~\cite{alouges2011, maxwell, multiscale, tran,  mathmod2012,gps}
which build on the work~\cite{alouges2008} of \textsc{Alouges}.
All of the above integrators have in common that they allow for constructive 
existence proofs of the corresponding problems. The idea, which is also 
exploited in the current work, is to show boundedness of the 
\emph{computable} discrete solutions.
Then, a compactness argument concludes the existence of weakly 
convergent subsequences. Finally, those weak limits are identified 
as weak solutions of the coupled system. 

In our contribution, we extend the analysis of the aforementioned 
works for the projection based integrators and show that these ideas can also
be transferred to the coupled system of LLG with magnetostriction. Even though 
the structure of the main proof in this work is similar to the one from 
e.g.~\cite{alouges2008,maxwell, multiscale} (and also the other works 
mentioned), we stress that the individual ingredients are 
much harder to gain than for the coupling with stationary 
equations~\cite{multiscale} 
or with the instationary Maxwell system~\cite{maxwell}. 
First, unlike the Maxwell-LLG system, the coupling to the conservation of 
momentum equation involves a nonlinear coupling operator. Second, the 
field contribution which accounts for magnetostriction, does not only depend 
on the displacement, but also on its spatial derivative.
For these two reasons, the analysis requires new mathematical tools 
and therefore complicates the convergence proof.
In contrast to the existing literature, we thus extend the approach 
from~\cite{alouges2008} and
analyze a nonlinear coupling of LLG to a second time-dependent PDE. The 
contributions of this work as well as the advances over the state of the art 
can be summarized as follows:
\begin{itemize}
\item We include the magnetostrictive field contribution into the numerical 
analysis as well as into the algorithm to account for elastic effects on a 
microscale. This allows to conduct more precise simulations in 
certain applications.
\item We give a new and constructive proof for the existence of 
weak solutions for LLG with magnetostriction by providing a numerical 
integrator which is mathematically guaranteed to converge to a weak solution 
of the coupled system of LLG with magnetostriction.

\item The proposed integrator is unconditionally convergent towards a weak 
solution as soon as timestep-size $k$ and spatial mesh-size $h$ tend to 
zero, independently of each other. For midpoint-scheme based 
integrators~\cite{banas, bp,rochat}, unconditional convergence is 
theoretically proved. However, the solution of the nonlinear system requires 
either a heuristical solver or an appropriate fixed-point iteration. The latter 
is used in~\cite{banas,bp,rochat}, but the resulting explicit scheme again 
involves a coupling of $h$ and $k$ to guarantee convergence and avoid 
instabilities.
\item The proposed algorithm is extremly attractive from a computational point 
of view: Since the two equations are fully decoupled, the implementation is 
easy and only two linear systems have to be solved per timestep. Contrary, 
midpoint-scheme based integrators~\cite{banas, bp,rochat} have to solve 
one large nonlinear system per timestep, and the decoupling has not been
thoroughly analyzed~\cite{banas}.
\item We provide a numerical comparison between the extended Alouges-type
integrator and an integrator based on the midpoint scheme~\cite{rochat}. 
Empirically, the numerical results show that our integrator works with 
larger timesteps than the midpoint scheme. In particular, in our setting
the computations based on the proposed integrator turn out to be much faster.
This gives empirical evidence that the fixed-point iteration for the
midpoint scheme indeed poses a computational bottleneck.
\end{itemize}

\textbf{Outline.}
The remainder of this paper is organized as follows: In Section~\ref{sec:problem}, we state the mathematical model for the coupled system of LLG with magnetostriction and recall the notion of a weak solution (Definition~\ref{def:weak_sol}). In Section~\ref{sec:prelim}, we collect some notation and preliminaries, as well as the definition of the discrete ansatz spaces. In Section~\ref{sec:algo}, we write down our numerical integrator in Algorithm~\ref{alg}, and
Section~\ref{sec:convergence} is devoted to our main convergence result (Theorem~\ref{thm:convergence}) and its proof. Finally, numerical examples concludes the work in Section~\ref{sec:numerics}.

\section{Model Problem}\label{sec:problem}
The evolution of the magnetization of a ferromagnetic body $\Omega$ during some time interval $(0,T)$ is mathematically modeled by the Landau-Lifshitz-Gilbert equation (LLG) which in dimensionless form reads
\begin{subequations}\label{eq:llg_total}
\begin{align}\label{eq:llg}
\mmm_t - \alpha \mmm\times \mmm_t = -\mmm\times \heff \quad \text{ in } \Omega_T := (0,T) \times \Omega
\end{align}
Here, $\mmm: \Omega_T \to \SSS^2 := \set{\xxx \in \R^3}{|\xxx| = 1}$ denotes the sought magnetization, and $\heff$ is the so-called effective field that consists of several energy contributions each of which models a certain micromagnetic effect. More precisely, in our work, the effective field consists of the exchange contribution $\Delta \mmm$, the magnetostrictive component $\hm$, and all other stationary and lower-order effects are collected in some general field contribution $\operator(\mmm)$, i.e.\ 
\begin{align}\label{eq:heff}
\heff = C_e\Delta \mmm + \hm - \operator(\mmm).
\end{align}
The general energy contribution $\operator(\mmm)$, is only assumed to fulfill a certain set of properties, see~\eqref{assumption:bounded}--\eqref{assumption:convergence}, and we emphasize that this particularly includes the case $\heff = C_e\Delta \mmm + \hm + \Ca D\Phi(\mmm) + P(\mmm) - \fff$. Here, $\Phi(\mmm)$ denotes the crystalline anisotropy density and $\fff$ is a given applied field. The contribution $P(\mmm)$ stands for the nonlocal strayfield. The constants $C_e, \Ca >0$ denote the exchange and anisotropy constant, respectively. To keep the presentation simple, we did not include the additional coupling to the full Maxwell equations as in~\cite{maxwell}. We stress, however, that this extension is straightforward and, with the combined techniques from this work and~\cite{maxwell}, one could consider a coupled system of full Maxwell LLG with
the conservation of momentum equation to account for magnetostrictive effects. The combined results from~\cite{maxwell} and the current work would then directly transfer to the coupled case and we would still derive unconditional convergence. Moreover, with the techniques from~\cite{multiscale}, it is straightforward to rigorously include a numerical approximation $\operator_h(\cdot)$ of $\operator(\cdot)$ into the convergence analysis.

For a bounded Lipschitz domain $\Omega \subset \R^3$ and a time interval $(0,T)$, we now aim to solve~\eqref{eq:llg} on $\Omega_T $ supplemented by the initial and boundary conditions
\begin{align}\label{eq:boundary}
\mmm(0) = \mmm^0 \in \H^1(\Omega; \SSS^2)Ê\quad \text{ and } \quad \partial_n\mmm = 0 \text{ on } (0,T) \times \partial \Omega.
\end{align}
\end{subequations}
The constraint $|\mmm^0| = 1$ almost everywhere in $\Omega$ models the fact that we consider a constant temperature below the Curie point.
Multiplication of~\eqref{eq:llg} with $\mmm$ yields $\partial_t|\mmm|^2 = 2\mmm\cdot \mmm_t = 0$. Hence, the modulus $|\mmm| = 1$ is preserved in time. By imposing the modulus constraint on $\mmm^0$, this guarantees $|\mmm(t)| = 1$ for almost all times $t \in (0,T)$. 

For modeling the magnetostrictive component, we follow the approach of \textsc{Visintin}~\cite{visintin}. Here, the magnetostrictive field reads
\begin{align}\label{eq:magnetostriction1}
\hm:\Omega_T \to \R^3, \quad
(\hm)_q = \big(\hm(\uuu, \mmm)\big)_q := \sum_{i,j,p=1}^{3}\lambda^m_{ijpq}\sigma_{ij}(\mmm)_p,
\end{align}
where $(\cdot)_p$ denotes the $p$-th component of a vector field.
We implicitly assume linear dependence of the stress tensor $\ssigma = \{\sigma_{ij}\}$ on the elastic part of the total strain $\eeps^e = \{\eps_{ij}^e\}$ which is the converse form of Hook's law, i.e.\
\begin{align}\label{eq:magnetostriction2}
 &\ssigma := \llambda^e\eeps^e(\uuu,\mmm) : \Omega_T \longrightarrow \R^{3\times3}, \quad
 \sigma_{ij} = \sum_{p,q=1}^3 \lambda^e_{ijpq}\eps^e_{pq},\\
 &\eeps^e(\uuu,\mmm) := \eeps(\uuu) - \eeps^m(\mmm) : [0,T]\times \Omega \longrightarrow \R^{3\times3}\label{eq:magnetostriction3},
\end{align}
where $\uuu : \Omega_T \rightarrow \R^3$ denotes the \emph{displacement vector field}. The \emph{total strain} is defined by the symmetric part of the gradient of $\uuu$, i.e.\
\begin{align}\label{eq:magnetostriction4}
 \eps_{ij}(\uuu) := \frac12\left(\frac{\partial u_i}{\partial x_j} + \frac{\partial u_j}{\partial x_i}\right),
\end{align}
and the magnetic part of the total strain by
\begin{equation}\label{eq:magnetostriction5}
\begin{split}
 \eeps^m(\mmm) := \llambda^m \mmm\mmm^T : \Omega_T \longrightarrow \R^{3\times3}, \quad
 \eps_{ij}^m(\mmm) = \sum_{p,q=1}^3 \lambda_{i j p q}^m(\mmm)_p(\mmm)_q.
\end{split}
\end{equation}
In addition, we assume both material tensors $\llambda \in \{\llambda^e, \llambda^m\}$ to be symmetric $(\lambda_{ijpq} = \lambda_{jipq} = \lambda_{ijqp} = \lambda_{pqij})$ and positive definite
\begin{align}\label{eq:lambda_posdef}
\sum_{i,j,p,q=1}^3 \lambda_{ijpq}\xi_{ij}\xi_{pq} \ge \lambda^\star\sum_{i,j=1}^3\xi^2_{ij}
\end{align}
with bounded entries, i.e.\ there exists some $\overline \lambda$ with $\lambda_{ijpq}^e, \lambda_{ijpq}^m \le \overline \lambda$ for any $i,j,p,q = 1,2,3$. 
The stress tensor $\ssigma$ and the displacement field $\uuu$ (where we assume no external forces) are finally coupled via the conservation of momentum equation
\begin{subequations}\label{eq:con_mom_total}
\begin{align}\label{eq:con_mom}
 \varrho \uuu_{tt} - \nabla \cdot \ssigma = 0 \quad \text{ in } \Omega_T.
\end{align}
Here, we assume the mass density $\varrho>0$ to be constant and independent of the deformation. Equation~\eqref{eq:con_mom} is additionally supplemented by the initial and boundary conditions
\begin{align}\label{eq:init_con_mom}
 \uuu(0)  = \uuu^0 \text{ in } \Omega, \quad
 \uuu_{t}(0) = \vvvv^0 \text{ in } \Omega, \quad \text{and} \quad
 \uuu = 0 \text{ on } \partial \Omega.
\end{align}
\end{subequations}
Altogether, we thus aim to solve the coupled problem
\begin{align}\label{eq:problem_total}
\begin{cases}
\mmm_t - \alpha \mmm\times \mmm_t = -\mmm\times \heff \\
 \varrho \uuu_{tt} - \nabla \cdot \ssigma = 0,
\end{cases}
\end{align}
subject to the stated initial and boundary conditions.

Using the above boundary conditions, Hook's relation~\eqref{eq:magnetostriction2}, the definition of the total strain tensor~\eqref{eq:magnetostriction3}, and the symmetry of the tensors $\llambda^e$ and $\llambda^m$, we obtain the following variational formulation of~\eqref{eq:con_mom} 
\begin{align}\label{eq:con_mom_var}
 \big(\varrho \uuu_{tt}(t), \vphi\big) + \big(\llambda^e\eeps(\uuu)(t),\eeps(\vphi)\big) = \big(\llambda^e\eeps^m(\mmm)(t), \eeps(\vphi)\big) \quad \text{ for all } \vphi \in \H^1_0(\Omega).
\end{align}
Given these notations, we now define our notion of a weak solution for the coupled LLG-magnetostriction system~\eqref{eq:problem_total}, which is the same as in~\cite{carbou}.
\begin{definition}\label{def:weak_sol}
The tupel $(\mmm, \uuu)$ is called a weak solution of LLG with magnetostriction, if for all $T > 0$,
\begin{enumerate}
\item[(i)] $\mmm \in \H^1(\Omega_T)$ with $|\mmm| = 1$ almost everywhere in $\Omega_T$ and $\uuu \in H^1(\Omega_T)$;
\item[(ii)] for all $\pphi \in C^\infty(\Omega_T)$ and $\zzeta \in C_c^\infty\big([0,T); C^\infty(\Omega) \big)$, we have
\begin{align}\nonumber
&\int_{\Omega_T} \langle\mmm_t, \pphi\rangle - \alpha \int_{\Omega_T} \langle(\mmm \times \mmm_t), \pphi\rangle = -\Ce \int_{\Omega_T} \langle(\nabla \mmm \times\mmm), \nabla\pphi\rangle\\
&\hspace*{3cm} + \int_{\Omega_T} \langle(\hm \times \mmm), \pphi\rangle - \int_{\Omega_T} \langle\big(\operator(\mmm)\times \mmm\big),\pphi\rangle\label{eq:weak_sol1}\\
&-\varrho\int_{\Omega_T}\langle\uuu_t, \zzeta_t\rangle + \int_{\Omega_T}\langle\llambda^e \eeps(\uuu),\eeps(\zzeta)\rangle
 = \int_{\Omega_T}\langle\llambda^e \eeps(\mmm), \eeps(\zzeta)\rangle + \int_\Omega \langle\vvvv^0, \zzeta(0, \cdot)\rangle\label{eq:weak_sol2}
\end{align}
\item[(iii)] there holds $\mmm(0,\cdot) = \mmm^0$ and $\uuu(0,\cdot) = \uuu_0$ in the sense of traces;
\item[(iv)] for almost all $t' \in (0,T)$, we have bounded energy
\begin{equation}\label{eq:energy}
\begin{split}
\norm{\nabla \mmm(t')}{\L^2(\Omega)}^2 + \norm{\mmm_t}{\L^2(\Omega_{t'})}^2 &+ \norm{\nabla \uuu(t')}{\L^2(\Omega)}^2  + \norm{\uuu_t(t')}{\L^2(\Omega)}^2\le
 \c{bounded},
\end{split}
\end{equation}
where $\setc{bounded}$ is independent of $t$ and depends only on $|\Omega|, \mmm_0, \uuu_0,$ and $\vvvv_0$.
\end{enumerate}
\end{definition}


\section{Preliminaries}\label{sec:prelim}
For time discretization, we impose a uniform partition $0 = t_0 < t_1 < \hdots < t_N = T$ of the time interval $[0,T]$. The timestep size is denoted by $k=k_j := t_{j+1} - t_j$ for $j= 0, \hdots, N-1$. For each (discrete) function $\vphi$ which is continuous in time, $\vphi^j = \vphi(t_j)$ denotes the evaluation at time $t_j$. 
For the time derivatives in the conservation of momentum equation~\eqref{eq:con_mom_var}, we use difference quotients of first and second order which are denoted by
\begin{align}\label{eq:diffquo}
\dt z_i = \frac{z_i - z_{i-1}}{k}, \qquad \dt^2 z_i = \frac{\dt z_i - \dt z_{i-1}}{k} = \frac{z_i -2z_{i-1}+z_{i-2}}{k^2}.
\end{align}

For spatial discretization, let $\TT_h$ be a quasi-uniform, regular triangulation of the polyhedral bounded Lipshitz domain $\Omega \subset \R^3$ into tetrahedra. The spatial mesh-size is denoted by $h$. By $\SS^1(\TT_h)$, we denote the lowest-order Courant FEM space of globally continuous and piecewise affine functions from $\Omega$ to $\R^3$, i.e. 
\begin{align*}
\SS^1(\TT_h) := \{\pphi_h \in C(\overline \Omega; \R^3) : \pphi_h|_K \in \PP_1(K) \text{ for all } K \in \TT_h\}.
\end{align*}
By $\II_h:C(\overline \Omega;\R^3)\to \SS^1(\TT_h)$, we denote the nodal interpolation operator onto this space.
The set of nodes of the triangulation $\TT_h$ is denoted by $\NN_h$.

For the discretization of the magnetization $\mmm$ in the LLG equation~\eqref{eq:llg}, we define the set of admissible discrete magnetizations by
\begin{align*}
\MM_h := \{\pphi_h \in \SS^1(\TT_h) : |\pphi_h(\zzz)| = 1 \text{ for all } \zzz \in \NN_h\}.
\end{align*}
Furthermore, for $\pphi_h \in \MM_h$, let
\begin{align*}
\KK_{\pphi_h} := \{\ppsi_h \in \SS^1(\TT_h) : \ppsi_h(\zzz) \cdot \pphi_h(\zzz) = 0 \text{ for all } \zzz \in \NN_h\}
\end{align*}
be the discrete tangent space associated with $\pphi_h$. Here, $\xxx\cdot \yyy$ stands for the usual Euclidean scalar product of $\xxx, \yyy \in \R^3$ which is sometimes also denoted by $\langle \cdot, \cdot \rangle$ to improve readability. The $L^2(\Omega)$ scalar product is denoted by $(\cdot, \cdot)$ throughout. Due to the modulus constraint $\mmm_t\cdot \mmm = 0$, and thus $|\mmm(t)|=1$ almost everywhere in $\Omega_T$, we discretize the time variable $\vvv(t_j) :=\mmm_t(t_j)$ in the discrete tangent space of $\mmm_h^j$.

To discretize the equation of magnetoelasticity~\eqref{eq:con_mom_var}, we employ $\SS^1_0(\TT_h) := \SS^1(\TT_h) \cap H^1_0(\Omega)$. In addition, let $\mmm_h^0 \in \MM_h$ and $\uuu_h^0, \dot \uuu_h^0 \in \SS^1_0(\TT_h)$ be suitable approximations of the initial data obtained e.g.\ by projection. Further requirements on those initial data are specified below in Theorem~\ref{thm:convergence}. Finally, we define $\dt \uuu_h^0$ as $\dot \uuu_h^0$. Througout this work, we write $A \lesssim B$ if there holds $A \le C B$ for some $h$ and $k$ independent constant $C>0$.
\section{Algorithm}\label{sec:algo}
For discretization of the LLG equation, we follow the approach of \textsc{Alouges}~\cite{alouges2008} which has been generalized in \cite{alouges2011} and \textsc{Goldenits et al.}~\cite{multiscale, petra,gps}. The main idea  is to introduce a new free variable $\vvv \approx \mmm_t$ and to interpret LLG as a linear equation in $\vvv$. This ansatz exploits the formulation
\begin{align}\label{eq:llg_equiv}
\alpha \mmm_t+\mmm \times \mmm_t = \heff - (\mmm\cdot \heff)\mmm,
\end{align}
which is equivalent to~\eqref{eq:llg} under the constraint $|\mmm| = 1$ almost everywhere, see e.g.\ \cite[Lemma 1.2.1]{petra}.
The conservation of momentum equation is discretized in space by a standard FEM approach, and by finite differences in time. This approximation is in analogy to the work of \textsc{Banas and Slodicka} \cite{banas06}, where the focus is on FEM discretizations for \emph{strong solutions} of~\eqref{eq:llg_total} with~\eqref{eq:con_mom_total}. In addition and for computational ease, the two equations can be decoupled. We propose and analyze the following algorithm:
\begin{algorithm}\label{alg}
Input: Initial data $\mmm_h^0$ and $\uuu_h^0$, parameter $0 \le \theta \le 1, \alpha >0$. For $\ell = 0,\hdots, N-1$ iterate:
\begin{enumerate}
\item[(i)] Compute unique solution $\vvv_h^\ell \in \KK_{\mmm_h^\ell}$ such that for all $\vphi_h \in \KK_{\mmm_h^\ell}$, we have
\begin{equation}\label{eq:alg1}
\begin{split}
&\alpha(\vvv_h^\ell, \vphi_h) + \big((\mmm_h^\ell \times \vvv_h^\ell), \vphi_h\big) = -C_e \big(\nabla(\mmm_h^\ell + \theta k \vvv_h^\ell), \nabla \vphi_h\big) \\&\hspace*{31ex}+ \big(\hm(\uuu_h^\ell, \mmm_h^\ell), \vphi_h\big) - \big(\operator(\mmm_h^\ell), \vphi_h\big).
\end{split}
\end{equation}
\item[(ii)] Define $\mmm_h^{\ell+1} \in \MM_h$ nodewise by $\mmm_h^{\ell+1}(\zzz) = \frac{\mmm_h^\ell(\zzz) + k\vvv_h^\ell(\zzz)}{|\mmm_h^\ell(\zzz) + k\vvv_h^\ell(\zzz)|}$ for all $\zzz \in \NN_h$.
\item[(iii)] Compute unique solution $\uuu_h^{\ell+1} \in \SS^1_0(\TT_h)$ such that for all $\ppsi_h \in \SS^1_0(\TT_h)$, we have
\begin{equation}\label{eq:alg2}
\begin{split}
\varrho (\dt^2 \uuu_h^{\ell+1}, \ppsi_h) + \big(\llambda^e \eeps(\uuu_h^{\ell+1}), \eeps(\ppsi_h)\big) = \big(\llambda^e \eeps^m(\mmm_h^{\ell+1}), \eeps(\ppsi_h)\big).
\end{split}
\end{equation}
\end{enumerate}
\end{algorithm}
In the above algorithm, the discrete magnetostrictive contribution is given by
\begin{align*}
[\hm(\uuu_h^\ell, \mmm_h^\ell)]_q := \sum_{i,j,p=1}^3\lambda^m_{ijpq}\sigma^h_{ij}(\mmm_h^\ell)_p, \quad \text{ with } \quad
\ssigma^h = \llambda^e\big(\eeps(\uuu_h^\ell) - \eeps^m(\mmm_h^\ell)\big).
\end{align*}

Exploiting that we solve each of the two equations separately, we can immediately state well-posedness of Algorithm~\ref{alg}.
\begin{lemma}\label{lem:solve}
Algorithm~\ref{alg} is well defined, i.e.\ it admits unique discrete solutions $(\vvv_h^\ell, \mmm_h^{\ell+1}, \uuu_h^{\ell+1})$ in each step $\ell = 0, \hdots, N-1$ of the iteration. Moreover, we have $\norm{\mmm_h^\ell}{\L^\infty(\Omega)} = 1$ for all $\ell = 1, \hdots, N$.
\end{lemma}
\begin{proof}
We first show solvability of~\eqref{eq:alg1}.
We define the bilinear form 
\begin{align*}
a_1^\ell(\cdot, \cdot) : \KK_{\mmm_h^\ell} \times \KK_{\mmm_h^\ell} \rightarrow \R, \quad
a_1^\ell(\bf \phi, \vphi) := \alpha(\pphi, \vphi) + \theta C_e k(\nabla \pphi, \nabla \vphi) + \big((\mmm_h^\ell \times \pphi),\vphi\big)
\end{align*}
and the linear functional
$L_1^\ell(\vphi) := C_e (\nabla \mmm_h^\ell, \nabla \vphi) + \big(\hm(\uuu_h^\ell, \mmm_h^\ell), \vphi\big) - \big(\operator(\mmm_h^\ell), \vphi\big).$
Then, \eqref{eq:alg1} is equivalent to
$a_1^\ell(\vvv_h^\ell, \vphi_h) = L_1^\ell(\vphi_h)\text{ for all } \vphi_h \in \KK_{\mmm_h^\ell}.$
Note that $a_1^\ell(\cdot, \cdot)$ is positive definite for $\alpha >0$, i.e.\ $a^\ell(\vphi, \vphi) \ge \alpha \norm{\vphi}{\L^2(\Omega)}^2$.
Thus, by exploiting finite dimension, we see that there exists a unique $\vvv_h^\ell \in \KK_{\mmm_h^\ell}$ which solves \eqref{eq:alg1}. 
Due to pointwise orthogonality of $\mmm_h^\ell$ and $\vvv_h^\ell$, and the Pythagoras theorem, we get $|\mmm_h^\ell(\zzz) + k\vvv_h^\ell(\zzz)|^2 = |\mmm_h^\ell(\zzz)|^2 + k|\vvv_h^\ell(\zzz)|^2\ge 1$ and thus even step $(ii)$ of the above algorithm is well-defined. The bound $\norm{\mmm_h^\ell}{\L^\infty(\Omega)} = 1$ can be seen by the normalization at the grid points in combination with barycentric coordinates and the convexity of each tetrahedron.

For the second equation \eqref{eq:alg2}, we consider the bilinear form
\begin{align*}
a_2(\cdot, \cdot) : \SS^1_0(\TT_h) \times \SS^1_0(\TT_h) \rightarrow \R,\quad
a_2(\zzeta,\ppsi) &:= \frac{\varrho}{k^2}(\zzeta, \ppsi) + \big(\llambda^e \eeps(\zzeta), \eeps(\ppsi)\big)
\end{align*}
and the linear functional
$L_2^\ell(\ppsi) = \big(\llambda^e \eeps^m(\mmm_h^{\ell+1}), \eeps(\ppsi)\big) + \frac{\varrho}{k}(\dt \uuu_h^\ell, \ppsi) + \frac{\varrho}{k^2}(\uuu_h^\ell, \ppsi),$
According to~\eqref{eq:lambda_posdef} and Korn's inequality~\cite[Thm.\ 11.2.16]{brennerscott}, it holds that $a_2(\ppsi, \ppsi) \gtrsim \norm{\ppsi}{\H^1(\Omega)}^2$. With this notation, \eqref{eq:alg2} is equivalent to
$a_2(\uuu_h^{\ell+1}, \ppsi_h) = L_2^\ell(\ppsi_h)$ for all functions $\ppsi_h \in \SS^1_0(\TT_h),$
and hence, admits a unique solution $\uuu_h^{\ell+1} \in \SS^1_0(\TT_h)$ in each step of the loop. 
\end{proof}


\section{Main Theorem}\label{sec:convergence}

In this section, we aim to show that the preceding algorithm indeed converges towards the correct limit. Before we start with the actual analysis, we collect some general assumptions and some more notation. Throughout, we assume that the spatial meshes $\TT_h$ are uniformly shape regular and satisfy the angle condition
\begin{align}\label{eq:assum1}
\int_\Omega \nabla \zeta_i \cdot \nabla \zeta_j \le 0 \quad \text{ for all hat functions } \zeta_i, \zeta_j \in \SS^1(\TT_h) \text{ with } i \neq j.
\end{align}
This somewhat technical condition is a crucial ingredient of the convergence proof, since it yields the discrete energy decay 
\begin{align}\label{eq:en_decay}
\norm{\nablaÊ\mmm_h^{\ell+1}}{\L^2(\Omega)}^2\le \norm{\nabla (\mmm_h^\ell + k \vvv_h^\ell)}{\L^2(\Omega)}^2.
\end{align}
The estimate~\eqref{eq:en_decay} is a direct consequence of the inequality $\norm{\nabla \II_h(\mmm/|\mmm|)}{\L^2(\Omega)}^2 \le \norm{\nabla \II_h \mmm}{\L^2(\Omega)}^2$ which was proved by \textsc{Bartels} in~\cite{bartels}. We like to emphasize that the angle condition is always fulfilled for tetrahedral meshes with dihedral angles that are smaller than $\pi/2$, cf.~\cite{bartels}, and easily preserved for uniform mesh-refinement.
For $\xxx \in \Omega$ and $t \in [t_\ell, t_{\ell+1})$ and for $\gamma_h^\ell \in \{\mmm_h^\ell, \vvv_h^\ell, \uuu_h^\ell\}$, we define the time approximations
\begin{align*}
&\wgamma(t, \xxx) := \frac{t-t_\ell}{k}\gamma_h^{\ell+1}(\xxx) + \frac{t_{\ell+1} - t}{k}\gamma_h^\ell(\xxx),\\
&\wgamma^-(t, \xxx):= \gamma_h^{\ell}(\xxx),\quad \wgamma^+(t,\xxx):= \gamma_h^{\ell+1}(\xxx).
\end{align*}
Note that $ \wgamma$ can also be written as
$\wgamma(t,\xxx) = \gamma_h^\ell(\xxx) + (t-t_{\ell})\dt \gamma_h^{\ell+1}(\xxx).$
In addition, for $t \in [t_\ell, t_{\ell+1})$, we define
\begin{align*}
\wvvvv(t, \xxx):= \dt \uuu_h^\ell(\xxx) + (t-t_\ell)\dt^2\uuu_h^{\ell+1}(\xxx), \quad \wvvvv^-(t, \xxx):=\dt \uuu_h^\ell(\xxx), \quad \wvvvv^+(t, \xxx):= \dt \uuu_h^{\ell+1}(\xxx).
\end{align*}
The next statement is the main theorem of this work and particularly includes the main result from~\cite{carbou}.

\begin{theorem}\label{thm:convergence}
$\mathbf{(a)}$ Let $\theta \in (1/2,1]$ and suppose that the meshes $\TT_h$ are uniformly shape regular and satisfy the angle condition~\eqref{eq:assum1}. Moreover, let the general energy contribution $\operator$ be uniformly bounded in $\L^2(\Omega_T)$, i.e.\
\begin{align}\label{assumption:bounded}
\norm{\operator(\nnn)}{\L^2(\Omega_T)}^2 \le \Cpi \quad \text{ for all } |\nnn| \in \L^2(\Omega_T) \text{ with } \nnn \le 1 \text{ a.e.\ in } \Omega_T,
\end{align}
with an $\nnn$-independent constant $\Cpi>0$ and assume weak convergence of the initial data, i.e.
$\mmm_h^0 \weakto \mmm^0, \uuu_h^0 \weakto \uuu^0$ in $\H^1(\Omega)$, as well as $\vvvv_h^0 \weakto \vvvv^0$ in $\L^2(\Omega)$ as $h \rightarrow 0$.
 Under these assumptions, we have strong $\L^2(\Omega_T)$-convergence of $\mmm_{hk}^-$ towards some function $\mmm \in \H^1(\Omega_T)$.

\bigskip

\noindent
$\mathbf{(b)}$ Suppose that, in addition to the above assumptions, we have
\begin{align}\label{assumption:convergence}
\operator(\mmm_{hk}^-) \weakto \operator(\mmm) \quad \text{ weakly subconvergent in } \L^2(\Omega_T).
\end{align}
Then, the computed FE solutions $(\mmm_{hk}, \uuu_{hk})$ are weakly subconvergent in $\H^1(\Omega_T)\times \H^1(\Omega_T)$ towards some functions $(\mmm, \uuu)$, and those weak limits $(\mmm, \uuu)$ are a weak solution of LLG with magnetostriction. In particular, weak solutions exist and 
each weak accumulation point of $(\mmm_{hk}, \uuu_{hk})$ is a weak solution in the sense of Definition~\ref{def:weak_sol}.
%
%
%
\end{theorem}

The proof will be done in roughly three steps.
\begin{enumerate}
\item[(i)] Boundedness of the discrete quantities and energies.
\item[(ii)] Existence of weakly convergent subsequences.
\item[(iii)] Identification of the limits with weak solutions of LLG with magnetostriction.
\end{enumerate}
As mentioned, we first show the desired boundedness and start with some preliminary lemmata.

\begin{lemma}\label{lem:h}
The discrete magnetostrictive component can be estimated by the total strain of the discrete displacement, i.e.\
\begin{align}\label{eq:h}
\norm{\hm(\uuu_h^\ell, \mmm_h^\ell)}{\L^2(\Omega)}^2 \le \c{hbounded1} \norm{\eeps(\uuu_h^\ell)}{\L^2(\Omega)}^2 + \c{hbounded2},
\end{align} 
for some constants $\setc{hbounded1}, \setc{hbounded2} > 0$ that depend only on $\overline \lambda$.
\end{lemma}
\begin{proof}
By definition of the magnetostrictive part, we immediately get
\begin{align*}
\norm{\hm(\uuu_h^\ell, \mmm_h^\ell)}{\L^2(\Omega)}^2 \lesssim\overline \lambda \norm{\mmm_h^\ell}{\L^\infty(\Omega)}^2\norm{\ssigma^h}{\L^2(\Omega)}^2 \lesssim \norm{\ssigma^h}{\L^2(\Omega)}^2
\end{align*}
due to the normalization step. From the definition of the discrete stress tensor and the boundedness of the material tensors, we additionally get
\begin{align*}
\norm{\ssigma^h}{\L^2(\Omega)}^2 \lesssim \overline{\lambda} \norm{\eeps(\uuu_h^\ell)}{\L^2(\Omega)}^2 + \norm{\eeps^m(\mmm_h^\ell)}{\L^2(\Omega)}^2 \lesssim \norm{\eeps(\uuu_h^\ell)}{\L^2(\Omega)}^2 + C.
\end{align*}
This yields the assertion.
\end{proof}

\begin{lemma}\label{lem:kvhj_bounded}
For $j = 1,\hdots, N$, there holds
\begin{equation}\label{eq:kvhj_bounded}
\begin{split}
\norm{\nabla \mmm_h^j}{\L^2(\Omega)}^2 + (\theta-\frac12)k^2\sum_{\ell=0}^{j-1}&\norm{\nabla \vvv_h^\ell}{\L^2(\Omega)}^2+ k\sum_{\ell=0}^{j-1}\norm{\vvv_h^\ell}{\L^2(\Omega)}^2\\
 &\le \c{vbounded1} \big(\norm{\nabla \mmm_h^0}{\L^2(\Omega)}^2 + k \sum_{\ell=0}^{j-1}\norm{\eeps(\uuu_h^\ell)}{\L^2(\Omega)}^2 +\c{vbounded2}\big),
\end{split}
\end{equation}
for constants $\setc{vbounded1}, \setc{vbounded2} > 0$ that depend only on $\Cpi$ as well as $\c{hbounded1}$ and $\c{hbounded2}$ from the previous lemma.
\end{lemma}
\begin{proof}
In~\eqref{eq:alg1}, we use the special test function $\vphi_h = \vvv_h^\ell \in \KK_{\mmm_h^\ell}$ and get
\begin{align*}
 \alpha (\vvv_h^\ell, \vvv_h^\ell) + (\underbrace{(\mmm_h^\ell \times \vvv_h^\ell), \vvv_h^\ell)}_{=0} = 
 - C_e \big(\nabla(\mmm_h^\ell + \theta k \vvv_h^\ell), \nabla \vvv_h^\ell\big) + \big(\hm(\uuu_h^\ell, \mmm_h^\ell), \vvv_h^\ell\big) - \big(\operator(\mmm_h^\ell),\vvv_h^\ell\big),
\end{align*}
whence
\begin{align*}
\alpha \norm{\vvv_h^\ell}{\L^2(\Omega)}^2 + C_e \theta\,k \norm{\nabla \vvv_h^\ell}{\L^2(\Omega)}^2 = -C_e(\nabla \mmm_h^\ell, \nabla \vvv_h^\ell) + \big(\hm(\uuu_h^\ell, \mmm_h^\ell), \vvv_h^\ell\big) - \big(\operator(\mmm_h^\ell), \vvv_h^\ell\big).
\end{align*}
Next, we use the fact that $\norm{\nabla \mmm_h^{\ell+1}}{\L^2(\Omega)}^2 \le \norm{\nabla(\mmm_h^\ell + k\vvv_h^\ell)}{\L^2(\Omega)}^2$, see~\eqref{eq:en_decay}, to obtain
\begin{equation}\label{eq:blub3}
\begin{split}
\frac12\norm{\nabla \mmm_h^{\ell+1}}{\L^2(\Omega)}^2 &\le \frac12\norm{\nabla \mmm_h^\ell}{\L^2(\Omega)}^2 + k (\nabla \mmm_h^\ell, \nabla \vvv_h^\ell) + \frac{k^2}{2}\norm{\nabla \vvv_h^\ell}{\L^2(\Omega)}^2\\
&\le \frac12 \norm{\nabla \mmm_h^\ell}{\L^2(\Omega)}^2 - \big(\theta - 1/2\big)k^2\norm{\nabla \vvv_h^\ell}{\L^2(\Omega)}^2\\
&\quad - \frac{\alpha k}{C_e}\norm{\vvv_h^\ell}{\L^2(\Omega)}^2 + \frac{k}{C_e}\big(\hm(\uuu_h^\ell, \mmm_h^\ell), \vvv_h^\ell\big) - \frac{k}{C_e}\big(\operator(\mmm_h^\ell), \vvv_h^\ell\big).
\end{split}
\end{equation}
We sum over the time intervals from $0$ to $j-1$, which yields for any $\nu > 0$
\begin{align*}
\frac12&\norm{\nabla \mmm_h^j}{\L^2(\Omega)}^2  + \frac{\alpha k}{C_e}\sum_{\ell=0}^{j-1}\norm{\vvv_h^\ell}{\L^2(\Omega)}^2 \\
&\le \frac12 \norm{\nabla \mmm_h^0}{\L^2(\Omega)}^2 
+ \frac{k}{C_e}\sum_{\ell=0}^{j-1}\big[\big(\hm(\uuu_h^\ell, \mmm_h^\ell), \vvv_h^\ell\big) - \big(\operator(\mmm_h^\ell), \vvv_h^\ell\big)\big] -\big(\theta - 1/2\big)k^2\sum_{\ell=0}^{j-1}\norm{\nabla \vvv_h^\ell}{\L^2(\Omega)}^2\\
&\le\frac12 \norm{\nabla \mmm_h^0}{\L^2(\Omega)}^2 
+ \frac{k}{4C_e \nu}\sum_{\ell=0}^{j-1}\big(\norm{\hm(\uuu_h^\ell, \mmm_h^\ell)}{\L^2(\Omega)}^2 + \norm{\operator(\mmm_h^\ell)}{\L^2(\Omega)}^2\big) \\
&\quad+ \frac{k \nu}{C_e}\sum_{\ell=0}^{j-1}\norm{\vvv_h^\ell}{\L^2(\Omega)}^2 - \big(\theta - 1/2\big)k^2\sum_{\ell=0}^{j-1}\norm{\nabla \vvv_h^\ell}{\L^2(\Omega)}^2 =: \text{RHS}.\\
\intertext{With Lemma~\ref{lem:h} and the uniform boundedness~\eqref{assumption:bounded} of $\operator(\cdot)$, we further obtain}
\text{RHS}&\lesssim  \frac12 \norm{\nabla \mmm_h^0}{\L^2(\Omega)}^2 
+\frac{k\c{hbounded1}}{4C_e \nu}\sum_{\ell=0}^{j-1}\big(\norm{\eeps(\uuu_h^\ell)}{\L^2(\Omega)}^2  + \c{hbounded2}\big) + \Cpi\\
&\quad+ \frac{k \nu}{C_e}\sum_{\ell=0}^{j-1}\norm{\vvv_h^\ell}{\L^2(\Omega)}^2 - \big(\theta - 1/2\big)k^2\sum_{\ell=0}^{j-1}\norm{\nabla \vvv_h^\ell}{\L^2(\Omega)}^2.
\end{align*}
In total we thus derive
\begin{align*}
\frac12\norm{\nabla \mmm_h^j}{\L^2(\Omega)}^2  &+ \frac{k}{C_e}(\alpha - \nu)\sum_{\ell=0}^{j-1}\norm{\vvv_h^\ell}{\L^2(\Omega)}^2 +(\theta-\frac12)k^2\sum_{\ell=0}^{j-1}\norm{\nabla \vvv_h^\ell}{\L^2(\Omega)}^2
\\ &\lesssim  \frac12 \norm{\nabla \mmm_h^0}{\L^2(\Omega)}^2 
+\frac{k\c{hbounded1}}{4C_e \nu}\sum_{\ell=0}^{j-1}\big(\norm{\eeps(\uuu_h^\ell)}{\L^2(\Omega)}^2 + \c{hbounded2}\big) + \Cpi
\end{align*}
Taking $\nu < \alpha$ thus yields the desired result.
\end{proof}

Given the last two lemmata, we now aim to show boundedness of the discrete quantities involved in equation~\eqref{eq:alg2}, i.e.\ boundedness  of the discrete displacement approximations. The following result has basically already been stated in~\cite[Lemma 3]{banas06}. Since we require some important modifications, 
we include the proof here.

\begin{proposition}\label{prop:u_bounded}
For any $j = 1, \hdots, N$ and $\frac12 \le \theta \le 1$, there holds
\begin{align}\label{eq:u_bounded}
 \norm{\dt \uuu_h^j}{\L^2(\Omega)}^2 + \sum_{\ell= 1}^j\norm{\dt \uuu_h^\ell - \dt \uuu_h^{\ell-1}}{\L^2(\Omega)}^2
+ \norm{\eeps(\uuu_h^j)}{\L^2(\Omega)}^2 + \sum_{\ell=1}^j\norm{\eeps(\uuu_h^\ell) - \eeps(\uuu_h^{\ell-1})}{\L^2(\Omega)}^2 \le \c{ubounded}
\end{align}
for some $h$ and $k$ independent constant $\setc{ubounded}>0$ which depends only on $\llambda^\star$ and the constants $\c{vbounded1}$ and $\c{vbounded2}$ from Lemma~\ref{lem:kvhj_bounded}.
\end{proposition}
\begin{proof}
We use $\ppsi_h = \uuu_h^{\ell+1} - \uuu_h^{\ell}$ as test function in~\eqref{eq:alg2} and sum up for $\ell=0, \hdots, j-1$ to see
\begin{align*}
\varrho(\dt \uuu_h^{\ell+1} - \dt \uuu_h^\ell, \dt \uuu_h^{\ell+1}) + \big(\llambda^e \eeps(\uuu_h^{\ell+1}), \eeps(\uuu_h^{\ell+1}) - \eeps(\uuu_h^\ell)\big) = \big(\llambda^e \eeps^m(\mmm_h^{\ell+1}), \eeps(\uuu_h^{\ell+1}) - \eeps(\uuu_h^\ell)\big).
\end{align*}
Recall that the Abel summation formula shows for any $v_\ell \in \R$ and $j \ge 0$
\begin{align*}
\sum_{\ell=0}^{j-1}(v_{\ell+1} - v_\ell, v_{\ell+1}) = \frac12|v_{j}|^2 - \frac12|v_0|^2 + \frac12\sum_{\ell=0}^{j-1}|v_{\ell+1}-v_\ell|^2.
\end{align*}
This in combination with the symmetry and the positive definiteness~\eqref{eq:lambda_posdef} of $\llambda^e$ now yields
\begin{align*}
  \norm{\dt \uuu_h^j}{\L^2(\Omega)}^2 &+ \sum_{\ell= 1}^j\norm{\dt \uuu_h^\ell - \dt \uuu_h^{\ell-1}}{\L^2(\Omega)}^2
+ \norm{\eeps(\uuu_h^j)}{\L^2(\Omega)}^2 + \sum_{\ell=1}^j\norm{\eeps(\uuu_h^\ell) - \eeps(\uuu_h^{\ell-1})}{\L^2(\Omega)}^2\\
&\le \widetilde C + C\sum_{\ell=1}^j\big(\llambda^e\eeps^m(\mmm_h^\ell), \eeps(\uuu_h^\ell) - \eeps(\uuu_h^{\ell-1})\big),
\end{align*}
for some generic constants $C, \widetilde C > 0$.
Note that the terms $\norm{\dt \uuu_h^0}{\L^2(\Omega)}^2 =\norm{\vvvv_h^0}{\L^2(\Omega)}^2$ and $\norm{\eeps(\uuu_h^0)}{\L^2(\Omega)}^2$ are uniformly bounded due to the assumed convergence of these initial data and are hidden in the constant $\widetilde C$.

Next, we rewrite the sum on the right-hand side as
\begin{align*}
 \sum_{\ell=1}^j&\big(\llambda^e\eeps^m(\mmm_h^\ell), \eeps(\uuu_h^\ell) - \eeps(\uuu_h^{\ell-1})\big)\\
&=(\llambda^e\eeps^m(\mmm_h^j), \eeps(\uuu_h^j)) - (\llambda^e\eeps^m(\mmm_h^1), \eeps(\uuu_h^0)) - \sum_{\ell=1}^{j-1}\big(\llambda^e\eeps^m(\mmm_h^{\ell+1}) - \llambda^e\eeps^m(\mmm_h^\ell), \eeps(\uuu_h^\ell)\big)\\
&=(\llambda^e\eeps^m(\mmm_h^j), \eeps(\uuu_h^j)) - (\llambda^e\eeps^m(\mmm_h^1), \eeps(\uuu_h^0)) - k\sum_{\ell=1}^{j-1}\big(\llambda^e\dt\eeps^m(\mmm_h^{\ell+1}), \eeps(\uuu_h^\ell)\big).
\end{align*}
For any $\eta> 0$ and due to boundedness of $\llambda^e$, we further get
\begin{align*}
  \norm{\dt \uuu_h^j}{\L^2(\Omega)}^2 &+ \sum_{\ell= 1}^j\norm{\dt \uuu_h^\ell - \dt \uuu_h^{\ell-1}}{\L^2(\Omega)}^2
+ \norm{\eeps(\uuu_h^j)}{\L^2(\Omega)}^2 + \sum_{\ell=1}^j\norm{\eeps(\uuu_h^\ell) - \eeps(\uuu_h^{\ell-1})}{\L^2(\Omega)}^2\\
& \lesssim 1 + k\sum_{\ell=1}^{j-1}\norm{\dt \eeps^m(\mmm_h^{\ell+1})}{\L^2(\Omega)}^2 + k\sum_{\ell=1}^{j-1}\norm{\eeps(\uuu_h^\ell)}{\L^2(\Omega)}^2 + \frac{1}{4\eta} \norm{\eeps^m(\mmm_h^j)}{\L^2(\Omega)}^2\\
& \quad+ \eta \norm{\eeps(\uuu_h^j)}{\L^2(\Omega)}^2 + \norm{\eeps^m(\mmm_h^1)}{\L^2(\Omega)}^2 +  \norm{\eeps(\uuu_h^0)}{\L^2(\Omega)}^2.
\end{align*}
For sufficiently small $\eta$, this can be simplified to
\begin{align*}
   \norm{\dt \uuu_h^j}{\L^2(\Omega)}^2 &+ \sum_{\ell= 1}^j\norm{\dt \uuu_h^\ell - \dt \uuu_h^{\ell-1}}{\L^2(\Omega)}^2
+ \norm{\eeps(\uuu_h^j)}{\L^2(\Omega)}^2 + \sum_{\ell=1}^j\norm{\eeps(\uuu_h^\ell) - \eeps(\uuu_h^{\ell-1})}{\L^2(\Omega)}^2\\
& \lesssim 1 + k\big(\sum_{\ell=1}^{j-1}\norm{\dt \eeps^m(\mmm_h^{\ell+1})}{\L^2(\Omega)}^2 + \sum_{\ell=1}^{j-1}\norm{\eeps(\uuu_h^\ell)}{\L^2(\Omega)}^2\big).
\end{align*}
Here, we again used convergence of the initial data.
Using the boundedness and symmetry of $\llambda^m$ in combination with boundedness of $\norm{\mmm_h^\ell}{\L^\infty(\Omega)}$, straightforward calculation shows
$\norm{\dt \eeps^m(\mmm_h^\ell)}{\L^2(\Omega)}^2 \lesssim \norm{\dt \mmm_h^\ell}{\L^2(\Omega)}^2,$
cf.\ \cite{diss}. In combination with
\begin{align*}
\norm{\dt \mmm_h^\ell}{\L^2(\Omega)}^2 = \norm{\frac{\mmm_h^\ell - \mmm_h^{\ell-1}}{k}}{\L^2(\Omega)}^2\le \norm{\vvv_h^{\ell-1}}{\L^2(\Omega)}^2,
\end{align*}
cf.\ \cite{alouges2008, alouges2011, multiscale} resp.~\cite[Lemma 3.3.2]{petra}, this results in
\begin{align*}
  \norm{\dt \uuu_h^j}{\L^2(\Omega)}^2 &+ \sum_{\ell= 1}^j\norm{\dt \uuu_h^\ell - \dt \uuu_h^{\ell-1}}{\L^2(\Omega)}^2
+ \norm{\eeps(\uuu_h^j)}{\L^2(\Omega)}^2 + \sum_{\ell=1}^j\norm{\eeps(\uuu_h^\ell) - \eeps(\uuu_h^{\ell-1})}{\L^2(\Omega)}^2\\
& \lesssim 1 + k\big(\sum_{\ell=1}^{j-1}\norm{\vvv_h^\ell}{\L^2(\Omega)}^2 + \sum_{\ell=1}^{j-1}\norm{\eeps(\uuu_h^\ell)}{\L^2(\Omega)}^2\big).
\end{align*}
Next, we apply Lemma~\ref{lem:kvhj_bounded} to see
\begin{align*}
  \norm{\dt \uuu_h^j}{\L^2(\Omega)}^2 &+ \sum_{\ell= 1}^j\norm{\dt \uuu_h^\ell - \dt \uuu_h^{\ell-1}}{\L^2(\Omega)}^2
+ \norm{\eeps(\uuu_h^j)}{\L^2(\Omega)}^2 + \sum_{\ell=1}^j\norm{\eeps(\uuu_h^\ell) - \eeps(\uuu_h^{\ell-1})}{\L^2(\Omega)}^2\\
& \lesssim 1 + \big(\norm{\nabla \mmm_h^0}{\L^2(\Omega)}^2 + k \sum_{\ell=0}^{j-1}\norm{\eeps(\uuu_h^\ell)}{\L^2(\Omega)}^2 + k\sum_{\ell=1}^{j-1}\norm{\eeps(\uuu_h^\ell)}{\L^2(\Omega)}^2\big)\\
&\lesssim 1 + k\sum_{\ell=0}^{j-1}\norm{\eeps(\uuu_h^\ell)}{\L^2(\Omega)}^2.
\end{align*}
Application of a discrete version of Gronwall's lemma finally yields the assertion.
\end{proof}

In order to show the desired $\H^1(\Omega_T)$-convergence of $\uuu$, we still need to show uniform boundedness of the $\L^2(\Omega_T)$-part. This result is stated in the next corollary.

\begin{corollary}\label{cor:u_l2_bounded}
Due to the boundary conditions employed, the Poincar\'e inequality in combination with Korn's inequality shows
\begin{align}\nonumber
\norm{\uuu_h^j}{\L^2(\Omega)}^2 + \sum_{\ell=1}^j\norm{\uuu_h^\ell - &\uuu_h^{\ell-1}}{\L^2(\Omega)}^2 \le \c{poincare}\big(\norm{\nabla \uuu_h^j}{\L^2(\Omega)}^2 + \sum_{\ell=1}^j\norm{\nabla(\uuu_h^\ell - \uuu_h^{\ell-1})}{\L^2(\Omega)}^2\big)\\
& \qquad \le \c{poincare}\big(\norm{\eeps (\uuu_h^j)}{\L^2(\Omega)}^2 + \sum_{\ell=1}^j\norm{\eeps(\uuu_h^\ell - \uuu_h^{\ell-1})}{\L^2(\Omega)}^2\big)\label{eq:u_l2_bounded}
\end{align}
for any $j = 1, \hdots, N$, where the constant $\setc{poincare}>0$ stems from Poincar\'e's inequality and thus only depends on the diameter of $\Omega$. According to Proposition~\ref{prop:u_bounded}, the right-hand side of~\eqref{eq:u_l2_bounded} is uniformly bounded. \qed
\end{corollary}

Proposition~\ref{prop:u_bounded} now immediately yields boundedness of the discrete magnetizations.

\begin{corollary}\label{cor:m_bounded}
For any $j = 1, \hdots, N$, there holds
\begin{align}\label{eq:m_bounded}
\norm{\nabla \mmm_h^j}{\L^2(\Omega)}^2 + k \sum_{\ell=0}^{j-1}\norm{\vvv_h^\ell}{\L^2(\Omega)}^2+ \big(\theta - \frac12\big)k^2\sum_{\ell=0}^{j-1}\norm{\nabla \vvv_h^\ell}{\L^2(\Omega)}^2 \le \c{mbounded}
\end{align}
for some constant $\setc{mbounded}>0$ which depends only on $\c{vbounded1}, \c{vbounded2}$, and $\c{ubounded}$.
\end{corollary}
\begin{proof}
From Lemma~\ref{lem:kvhj_bounded}, we get
\begin{align}
\begin{split}
\norm{\nabla \mmm_h^j}{\L^2(\Omega)}^2 + (\theta-\frac12)k^2\sum_{\ell=0}^{j-1}&\norm{\nabla \vvv_h^\ell}{\L^2(\Omega)}^2+ k\sum_{\ell=0}^{j-1}\norm{\vvv_h^\ell}{\L^2(\Omega)}^2\\
 &\le \c{vbounded1} \big(\norm{\nabla \mmm_h^0}{\L^2(\Omega)}^2 + k \sum_{\ell=0}^{j-1}\norm{\eeps(\uuu_h^\ell)}{\L^2(\Omega)}^2 +\c{vbounded2}\big),
\end{split}
\end{align}
By utilizing Proposition~\ref{prop:u_bounded}, we see
$k \sum_{\ell=0}^{j-1}\norm{\eeps(\uuu_h^\ell)}{\L^2(\Omega)}^2 \le |T|\c{ubounded}.$
The uniform boundedness of $\norm{\nabla \mmm_h^0}{\L^2(\Omega)}$ concludes the proof.
\end{proof}

Next, we deduce the existence of convergent subsequences.
\begin{lemma}\label{lem:subsequences}
Let $1/2 \le  \theta \le 1$. Then, there exist functions $(\mmm, \uuu, \vvvv) \in \H^1(\Omega_T, \SSS^2)\times \H^1(\Omega_T) \times \L^2(\Omega_T)$ such that
\begin{subequations}
\begin{align}
& \mmm_{hk} \rightharpoonup \mmm \text{ in } \H^1(\Omega_T)\label{eq:subsequences1b}\\
&\mmm_{hk}, \mmm_{hk}^\pm \rightharpoonup \mmm \text{ in } L^2(\H^1),\label{eq:subsequences1a}\\
&\mmm_{hk}, \mmm_{hk}^\pm \rightarrow \mmm \text{ in } \L^2(\Omega_T)\label{eq:subsequences2}\\
&\mmm_{hk}, \mmm_{hk}^\pm \to \mmm \text{ pointwise almost everywhere in } \Omega_T, \label{eq:subsequences2b}\\
&\uuu_{hk} \rightharpoonup \uuu \text{ in } \H^1(\Omega_T)\label{eq:subsequences3b}\\
&\uuu_{hk}, \uuu_{hk}^\pm  \rightharpoonup \uuu \text{ in } L^2(\H^1),\label{eq:subsequences3a}\\
&\uuu_{hk}, \uuu_{hk}^\pm \rightarrow \uuu \text{ in } \L^2(\Omega_T)\label{eq:subsequences4}\\
&\wvvvv, \wvvvv^\pm \rightharpoonup \vvvv \text{ in } \L^2(\Omega_T)\label{eq:subsequences5}.
\end{align}
\end{subequations}
Here, the convergence is to be understood for a subsequence of the corresponding sequences which is successively constructed, i.e.\ for arbitrary spatial mesh-size $h \to 0$ and timestep-size $k \to 0$, there exist subindices $h_n, k_n$, for which the above convergence properties are satisfied simultaneously.
In addition, there exists some $\vvv \in \L^2(\Omega_\tau)$ with 
\begin{align}
\vvv_{hk}^- \rightharpoonup \vvv \text{ in } \L^2(\Omega_T)
\end{align}
again for the same subsequence as above.
\end{lemma}
\begin{proof}
From Proposition~\ref{prop:u_bounded}, Corollary~\ref{cor:u_l2_bounded}, and Corollary~\ref{cor:m_bounded}, we immediately get boundedness of all of those sequences. A compactness argument thus allows to successively extract convergent subsequences. Therefore, it only remains to show, that the corresponding limits coincide, i.e.\
\begin{align*}
\lim \wgamma = \lim \wgamma^- = \lim \wgamma^+ \quad \text{ with } \gamma_{hk} \in \{\mmm_{hk}, \uuu_{hk}, \vvvv_{hk}\}.
\end{align*}
To see this, we make use of the boundedness of two subsequent solutions. From Proposition~\ref{prop:u_bounded}, Corollary~\ref{cor:u_l2_bounded}, and Corollary~\ref{cor:m_bounded}, we see that
\begin{align*}
\sum_{\ell = 0}^{j-1}\norm{\mmm_h^{\ell+1} - \mmm_h^\ell}{\L^2(\Omega)}^2 &+ 
\sum_{\ell=0}^{j-1}\norm{\uuu_h^{\ell+1} - \uuu_h^\ell}{\L^2(\Omega)}^2\\
&\quad + \sum_{\ell=0}^{j-1}\norm{\eeps(\uuu_h^{\ell+1})-\eeps(\uuu_h^\ell)}{\L^2(\Omega)}^2 +  \sum_{\ell=0}^{j-1}\norm{\dt(\uuu_h^{\ell+1})-\dt(\uuu_h^\ell)}{\L^2(\Omega)}^2 
\end{align*}
is uniformly bounded. For the first sum, we used the inequality
\begin{align*}
\norm{\mmm_h^{\ell+1} - \mmm_h^\ell}{\L^2(\Omega)}^2 \le k^2\norm{\vvv_h^\ell}{\L^2(\Omega)}^2,
\end{align*}
see e.g.~\cite{alouges2008, petra}.
We thus get
\begin{align*}
\norm{\wgamma - \wgamma^-}{\L^2(\Omega_T)}^2 &=\sum_{\ell=0}^{N-1} \int_{t_\ell}^{t_{\ell+1}}\norm{\gamma_h^\ell + \frac{t-t_\ell}{k}(\gamma_h^{\ell+1} - \gamma_h^\ell) - \gamma_h^\ell}{\L^2(\Omega)}^2\\
&\le k\sum_{\ell=0}^{N-1}\norm{\gamma_h^{\ell+1}-\gamma_h^\ell}{\L^2(\Omega)}^2 \xrightarrow{h \to 0} 0,
\end{align*}
and analogously
\begin{align*}
\norm{\wgamma - \wgamma^+}{\L^2(\Omega_T)}^2 &=\sum_{\ell=0}^{N-1}\int_{t_\ell}^{t_{\ell+1}}\norm{\gamma_h^\ell + \frac{t-t_\ell}{k}(\gamma_h^{\ell+1} - \gamma_h^\ell) - \gamma_h^{\ell+1}}{\L^2(\Omega)}^2\\
&\le \sum_{\ell=0}^{N-1}\int_{t_\ell}^{t_{\ell+1}}2\norm{\gamma_h^{\ell+1} - \gamma_h^\ell}{\L^2(\Omega)}^2\\
&\le2k\sum_{\ell=0}^{N-1}\norm{\gamma_h^{\ell+1} - \gamma_h^\ell}{\L^2(\Omega)}^2 \xrightarrow{h \to 0} 0.
\end{align*}
We thus conclude that the limits $\lim \wgamma = \lim \wgamma^\pm$ coincide in $\L^2(\Omega_T)$. From the continuous inclusions $\H^1(\Omega_T) \Subset L^2(\H^1) \subseteq \L^2(\Omega_T)$ and the uniqueness of weak limits, we even conclude the convergence properties in $L^2(\H^1)$, resp.\ in $\H^1(\Omega_T)$. By use of the Weyl theorem, we may extract yet another subsequence to see pointwise convergence, i.e.~\eqref{eq:subsequences2b}. From
\begin{align*}
&\norm{|\mmm|-1}{\L^2(\Omega_T)}\le \norm{|\mmm|-|\mmm_{hk}^-|}{\L^2(\Omega_T)} + \norm{|\mmm_{hk}^-|-1}{\L^2(\Omega_T)}, \quad \text{and}\\
&\norm{|\mmm_{hk}^-(t, \cdot)|-1}{\L^2(\Omega)} \le h \max_{t_j}\norm{\nabla \mmm_h^\ell}{\L^2(\Omega)} \xrightarrow{h \to 0} 0,
\end{align*}
we finally conclude $|\mmm| = 1$ almost everywhere in $\Omega_T$, which is the desired result.
\end{proof}

In the remainder of this chapter, we will prove that the limiting tupel $(\mmm, \uuu)$ is indeed a weak solution  in the sense of Definition~\ref{def:weak_sol}. First, we identify the limit function $\vvv$ with the time derivative of $\mmm$. The following result can be found e.g.\ in~\cite{alouges2008}.

\begin{lemma}\label{lem:v=dtm}
The limit function $\vvv \in \L^2(\Omega_T)$ equals the time derivative of $\mmm$, i.e.\ $\vvv = \partial_t \mmm$ almost everywhere in $\Omega_T$. \qed
\end{lemma}

We have now collected all ingredients for the proof of our main theorem.

\emph{Proof of Theorem~\ref{thm:convergence}.}
Let $(\zzeta, \ppsi) \in C^\infty(\Omega_T) \times C_c^\infty\big([0,T);C^\infty(\Omega)\big)$ be arbitrary. We  define testfunctions by $(\vphi_h, \ppsi_h)(t, \cdot) := \big(\II_h(\mmm_{hk}^-\times \zzeta), \II_h\ppsi, \big)(t,\cdot)$. With the notation from above, we integrate equation~\eqref{eq:alg1} in time to obtain
\begin{align*}
\alpha \int_0^T (\vvv_{hk}^-, \vphi_h) + \int_0^T\big((\mmm_{hk}^- &\times \vvv_{hk}^-), \vphi_h\big) = -C_e\int_0^T\big(\nabla(\mmm_{hk}^- + \theta k \vvv_{hk}^-), \nabla \vphi_h)\big)\\
&\quad + \int_0^T \big(\hm(\uuu_{hk}^-, \mmm_{hk}^-), \vphi_h\big) - \int_0^T\big(\operator(\mmm_{hk}^-), \vphi_h\big).
\end{align*}
Then the magnetostrictive component is again given by
\begin{align*}
[\hm(\uuu_{hk}^-, \mmm_{hk}^-)]_\ell := \sum_{i,j,p}\lambda^m_{ijpq}\sigma^{hk}{ij}(\mmm_{hk}^-)_p, \quad \text{ with } \quad
\ssigma^{hk} = \llambda^e\big(\eeps(\uuu_{hk}^-) - \eeps^m(\mmm_{hk}^-)\big).
\end{align*}
The definition $\vphi_h (t, \cdot) := \II_h(\mmm_{hk}^- \times \zzeta)(t, \cdot)$ and the approximation properties of the nodal interpolation operator show
\begin{align*}
\int_0^T \big((\alpha \vvv_{hk}^- + \mmm_{hk}^- \times\vvv_{hk}^-),&(\mmm_{hk}^- \times \zzeta)\big) + k\,\theta\int_0^T \big(\nabla \vvv_{hk}^-, \nabla(\mmm_{hk}^- \times \zzeta)\big)\\
&\quad+ C_e\int_0^T\big(\nabla \mmm_{hk}^-, \nabla(\mmm_{hk}^- \times \zzeta)\big)\\
&\quad-\int_0^T \big(\hm(\uuu_{hk}^-, \mmm_{hk}^-), (\mmm_{hk}^-\times \zzeta)\big)\\
&\quad + \int_0^T \big(\operator(\mmm_{hk}^-), (\mmm_{hk}^-\times \zzeta)\big)\\
&=\mathcal{O}(h).
\end{align*}
Passing to the limit and using the strong $\L^2(\Omega_T)$-convergence of $(\mmm_{hk}^-\times \zzeta)$ towards $(\mmm \times \zzeta)$, we get
\begin{align*}
\int_0^T \big((\alpha \vvv_{hk}^- + \mmm_{hk}^- \times\vvv_{hk}^-),(\mmm_{hk}^- \times \zzeta)\big) &\longrightarrow \int_0^T \big((\alpha \mmm_t + \mmm \times\mmm_t),(\mmm \times \zzeta)\big),\\
k\,\theta\int_0^T \big(\nabla \vvv_{hk}^-, \nabla(\mmm_{hk}^- \times \zzeta)\big) &\longrightarrow 0, \quad \text{ and }\\\int_0^T\big(\nabla \mmm_{hk}^-, \nabla(\mmm_{hk}^- \times \zzeta)\big) &\longrightarrow \int_0^T\big(\nabla \mmm, \nabla(\mmm \times \zzeta)\big),
\end{align*}
as $(h,k)Ê\rightarrow (0,0)$, cf.~\cite[Proof of Thm.\ 2]{alouges2008}. Here, we have used the boundedness of $k\norm{\nabla \vvv_{hk}^-}{\L^2(\Omega_t)}^2$ which follows from $\theta \in (1/2,1]$, see Corollary~\ref{cor:m_bounded}. Next, the weak convergence of $\operator(\mmm_{hk}^-)$ from~\eqref{assumption:convergence} yields
\begin{align*}
\int_0^T \big(\operator(\mmm_{hk}^-), (\mmm_{hk}^-\times \zzeta)\big) &\longrightarrow \int_0^T \big(\operator(\mmm), (\mmm\times \zzeta)\big).
\end{align*}
As for the magnetostrictive component, we have to show 
\begin{align*}
\hm(\uuu_{hk}^-, \mmm_{hk}^-) \rightharpoonup \hm(\uuu, \mmm) \quad \text{ in } \L^2(\Omega_T),
\end{align*}
where it obviously suffices to show the desired property componentwise.
With the definition from~\eqref{eq:magnetostriction5} and the boundedness of $\llambda^m$, a direct computation proves
\begin{align*}
\norm{\eeps^m(\mmm_{hk}^-) - \eeps^m(\mmm)}{\L^2(\Omega_T)}^2 \lesssim \norm{\mmm_{hk}^- - \mmm}{\L^2(\Omega_T)}^2,
\end{align*}
cf.\ \cite{diss}. The pointwise convergence of $\mmm_{hk}^-$ from~\eqref{eq:subsequences2b}  in combination with Lebesgue's dominated convergence theorem, now yield the strong convergence $\mmm_{hk}^- \cdot \zzeta \to \mmm \cdot \zzeta$. For any indices $i,j,p = 1,2,3$ this shows
\begin{align*}
\big((\eeps^m(\mmm_{hk}^-))_{ij}(\mmm_{hk}^-)_p, \zzeta_p\big) \to \big(\eeps^m(\mmm)_{ij}\mmm_p, \zzeta_p\big).
\end{align*} 
By definition of the magnetostrictive component it therefore only remains to show
\begin{align*}
(\eeps(\uuu_{hk}^-))_{ij} (\mmm_{hk}^-)_p \rightharpoonup \eeps(\uuu)_{ij}\cdot \mmm_p \quad \text{ in } \L^2(\Omega_T)
\end{align*}
for any combination $i,j,p = 1,2,3$ of indices.
Analogously to above, this can be seen by
\begin{align*}
\big((\eeps(\uuu_{hk}^-))_{ij} (\mmm_{hk}^-)_p, \zzeta_p \big) = \big(\eeps(\uuu_{hk}^-)_{ij}, (\mmm_{hk}^-)_p\zzeta_p\big) \rightarrow \big(\eeps(\uuu)_{ij}, \mmm_p \zzeta_p\big) = \big(\eeps(\uuu)_{ij}\mmm_p, \zzeta_p\big)
\end{align*}
for all $\zzeta \in C^\infty(\Omega_T)$, where the convergence of $\eeps(\uuu_{hk}^-)$ towards $\eeps(\uuu)$ particularly follows from the convergence of $\uuu_{hk}^-$ towards $\uuu$ in $L^2(\H^1)$. So far, we have thus proved
\begin{align*}
\int_0^T \big((\alpha \mmm_t + \mmm \times \mmm_t), (\mmm \times \zzeta)\big) &= -C_e \int_0^T \big(\nabla \mmm, \nabla (\mmm \times \zzeta)\big)\\
&\quad + \int_0^T \big( \hm(\uuu, \mmm), (\mmm \times \zzeta)\big) - \int_0^T \big(\operator(\mmm),(\mmm \times \zzeta)\big).
\end{align*}
Proceeding as in~\cite{maxwell}, we conclude~\eqref{eq:weak_sol1} by use of elementary pointwise calculations. The equality $\mmm(0,\cdot) = \mmm^0$ in the trace sense follows from the weak convergence $\mmm_{hk} \weakto \mmm$ in $\H^1(\Omega_T)$ and thus weak convergence of the traces. The equality $\uuu(0,\cdot) = \uuu^0$ follows analogously.

To prove~\eqref{eq:weak_sol2}, we argue similarly. From~\eqref{eq:alg2}, we obtain
\begin{align*}
&\int_0^T \big((\wvvvv)_t, \ppsi_h\big) + \int_0^T\big(\llambda^e\eeps(\uuu_{hk}^+), \eeps(\ppsi_h)\big) = \int_0^T \big(\llambda^e\eeps^m(\mmm_{hk}^+), \eeps(\ppsi_h)\big).
\end{align*}
For the first summand on the left-hand side, we perform integration by parts in time and, due to the shape of $\ppsi_h$, get
\begin{align*}
\int_0^T \big((\wvvvv)_t, \ppsi_h\big) = -\int_0^T\big(\wvvvv, (\ppsi_h)_t\big) + \underbrace{\big(\wvvvv(T,\cdot), \ppsi_h(T, \cdot)\big)}_{=0} - \big(\underbrace{\wvvvv(0,\cdot)}_{= \vvvv_h^0}, \ppsi_h(0,\cdot)\big).
\end{align*}
Passing to the limit $(h,k) \to 0$, we see
\begin{align*}
\int_0^T \big((\wvvvv)_t, \ppsi_h\big) \longrightarrow -\int_0^T\big(\vvvv, \ppsi_t\big) - \big(\vvvv(0,\cdot), \ppsi(0,\cdot)\big).
\end{align*}
Here, we have used the assumed convergence of the initial data. From~\eqref{eq:subsequences5} and the definition of $\vvvv_{hk}^+$, we get $\wvvvv^+ = \partial_t \uuu_{hk}$, and therefore, by use of weak lower semi-continuity, conclude
\begin{align*}
\norm{\vvvv - \partial_t \uuu}{\L^2(\Omega)}^2 \le \liminf \norm{\wvvvv^+ - \partial_t \uuu_{hk}}{\L^2(\Omega)}^2 = 0
\end{align*}
whence $\vvvv = \partial_t \uuu$ almost everywhere in $\Omega_T$.
The convergence of the terms
\begin{align*}
 \int_0^T\big(\llambda^e\eeps(\uuu_{hk}^+), \eeps(\ppsi_h)\big) & \longrightarrow  \int_0^T\big(\llambda^e\eeps(\uuu), \eeps(\ppsi)\big) \text{ and }\\
\int_0^T \big(\llambda^e\eeps^m(\mmm_{hk}^+), \eeps(\ppsi_h)\big) & \longrightarrow \int_0^T \big(\llambda^e\eeps^m(\mmm), \eeps(\ppsi)\big)
\end{align*}
is straightforward. In summary, we have thus shown~\eqref{eq:weak_sol2}. 


It remains to show the energy estimate~\eqref{eq:energy}. From the discrete energy estimates~\eqref{eq:u_bounded} and~\eqref{eq:m_bounded}, in combination with Korn's inequality, we get for any $t' \in [0,T]$ with $t' \in [t_\ell, t_{\ell+1})$
\begin{align*}
&\norm{\nabla \mmm_{hk}^+(t')}{\L^2(\Omega)}^2 +  \norm{\vvv_{hk}^-}{\L^2(\Omega_{t'})}^2 + \norm{\nabla \uuu_{hk}^+(t')}{\L^2(\Omega)}^2+\norm{\wvvvv^+(t')}{\L^2(\Omega)}^2\\
&=\norm{\nabla \mmm_{hk}^+(t')}{\L^2(\Omega)}^2 + \int_0^{t'} \norm{\vvv_{hk}^-(t)}{\L^2(\Omega)}^2 + \norm{\nabla \uuu_{hk}^+(t')}{\L^2(\Omega)}^2+ \norm{\wvvvv^+(t')}{\L^2(\Omega)}^2\\
&\le \norm{\nabla \mmm_{hk}^+(t')}{\L^2(\Omega)}^2 + \int_0^{t_{\ell+1}} \norm{\vvv_{hk}^-(t)}{\L^2(\Omega)}^2 +\norm{\nabla \uuu_{hk}^+(t')}{\L^2(\Omega)}^2+ \norm{\wvvvv^+(t')}{\L^2(\Omega)}^2 \\
&\le \widetilde C,
\end{align*}
for some constant $\widetilde C$ which is independent of $h$ and $k$.
Integration in time thus yields for any measurable set $\TTTT \subseteq [0,T]$
\begin{align*}
\int_{\TTTT}\norm{\nabla \mmm_{hk}^+(t')}{\L^2(\Omega)}^2 + \int_{\TTTT}\norm{\vvv_{hk}^-}{\L^2(\Omega_{t'})}^2 + \int_{\TTTT}\norm{\nabla \uuu_{hk}^+(t')}{\L^2(\Omega)}^2 + \int_{\TTTT}\norm{\wvvvv^+(t')}{\L^2(\Omega)}^2 \le \int_\TTTT \widetilde C.
\end{align*}
Weak semi-continuity now shows
\begin{align*}
&\int_{\TTTT}\norm{\nabla \mmm}{\L^2(\Omega)}^2 + \int_{\TTTT}\norm{\mmm_t}{\L^2(\Omega_{t'})}^2 + \int_{\TTTT}\norm{\nabla \uuu}{\L^2(\Omega)}^2 + \int_{\TTTT}\norm{\uuu_t}{\L^2(\Omega)}^2 \le \int_\TTTT \widetilde C,
\end{align*}
which concludes the proof.
\qed

\begin{remark}
Finally, we like to comment on the choice of $0 \le \theta\le 1/2$ to emphasize how variants of Theorem~\ref{thm:convergence} are read and proved.
\begin{enumerate}
\item For $0 \le \theta < 1/2$, one has to bound the term $k^2 \sum_{\ell=0}^{j-1} \norm{\nabla \vvv_h^\ell}{\L^2(\Omega)}^2$ in Corollary~\ref{cor:m_bounded} in order to prove boundedness of the discrete quantities. This can be achieved by using an inverse estimate $\norm{\nabla \vvv_h^\ell}{\L^2(\Omega)}^2 \lesssim \frac{1}{h^2} \norm{\vvv_h^\ell}{\L^2(\Omega)}^2$. The latter term can then be absorbed into the term $k \sum_{\ell = 0}^{j-1}\norm{\vvv_h^\ell}{\L^2(\Omega)}^2$ and thus yields convergence as $\frac{k}{h^2} \rightarrow 0$.
\item For the intermediate case $\theta = \frac12$, the limit
$k \theta \int_0^T \big(\nabla \vvv_{hk}^-, \nabla(\mmm_{hk}^- \times \zzeta)\big) \to 0$
is no longer valid, see~\cite[Proof of Thm.~2]{alouges2008}. As suggested in~\cite{alouges2008}, this can be circumvented by an inverse estimate provided that $\frac{k}{h} \to 0$.
\end{enumerate}
\end{remark}

\section{Numerical experiments}\label{sec:numerics}

In this section, we investigate the performance of the proposed Algorithm~\ref{alg} empirically. We compare Algorithm~\ref{alg} with the midpoint scheme from
\cite{bp, rochat} for the following blow-up benchmark example in 2D, proposed in \cite{bartels,bjp}: On $\Omega=(-0.5,0.5)^2$, we solve problem~\eqref{eq:problem_total}. For some fixed parameter $s \in \R$ and $ A := (1-2|\textbf{x}|)^4/s$ , the initial magnetization reads
%
\[
\textbf{m}_0(\textbf{x})= \left\{ \begin{array}{l l }
(0,0,-1) & \text{for } | \textbf{x} | \geq 0.5 \\[2mm]
\frac{(2\textbf{x}A, A^2 - | \textbf{x} |^2)}{(A^2 + | \textbf{x} |^2)} & \text{for } | \textbf{x} | \leq 0.5.
\end{array}\ \right.
\]
This choice of $\mmm_0$ is motivated in~\cite{bjp} in order to form some singularity at the center $\xxx = 0$.

The triangulation $\mathcal{T}_r$ used in the numerical simulation is defined through a positive integer $r$ and consists of $2^{2r+1}$ halved squares with edge length $h=2^{-r}$. Since this triangulation is of Delaunay type, the angle condition \eqref{assumption:bounded} is satisfied, see~\cite{alouges2008}.
For the magnetostriction, we took $\boldsymbol{\lambda}^e$ and  $\boldsymbol{\lambda}^m$ to be $2 \times 2$ tensors with
$
\boldsymbol{\lambda}^e=\lambda_{1111}^e=\lambda_{2222}^e=C^e, \quad \boldsymbol{\lambda}^m= \lambda_{1111}^m= \lambda_{2222}^m=C^m.
$
for given constants $C^e, C^m\ge0$.
For comparison, the midpoint scheme is computed as described in \cite{rochat}. The latter requires a fixed point iteration, where the tolerance is chosen as $\varepsilon=10^{-10}$, and a time-step $k_m >0$ that satisfies the associated mesh size condition $k_m \le Ch^2$ from~\cite{banas, bp, rochat}. In our computations for the midpoint scheme, we thus chose $k_m=h^2/10$. The time-step size of Algorithm~\ref{alg} is denoted by $k_a$ and does not depend on the spatial mesh-size, i.e.~as $h$ is decreased, $k_a$ remains fixed throughout the computations while $k_m$ needs to be stepwise decreased. The linear systems of Algorithm \ref{alg} are solved using an iterative method, and the constraint on the space $\KK_{\mmm_h^\ell}$ is incorporated via the Lagrange multiplier approach from~\cite{mathmod2012, petra}.  

\begin{table}[h]
\begin{center}
   \begin{tabular}{| c | c | c | c | c | c | c | c | c | }
     \hline
    \multicolumn{1}{|c|}{$\displaystyle$ $1/h$}   & \multicolumn{2}{|c|}{16} & \multicolumn{2}{|c|}{32} & \multicolumn{2}{|c|}{64} & \multicolumn{2}{|c|}{96} \\
       & CPU & $T_B$ & CPU & $T_B$  & CPU & $T_B$ & CPU & $T_B$ \\ \hline $\displaystyle$
    Alg \ref{alg} &  47.5&  0.036 & 291.2 & 0.024 & 2506 & 0.04 & 14157 & 0.059   \\ 
    Mid  & 136.5 & 0.038 &  1267 & 0.022 &  20371 & 0.032 & N.A.N & N.A.N  \\ \hline 
   \end{tabular}
   \caption{Comparison of CPU and blow-up time $T_B$ for Algorithm \ref{alg} and midpoint scheme \cite{bp} with $\alpha=1$, T=0.3[s], and $ C^e=C^m=0.$ \label{blowup}}
   \end{center}
\end{table}

\begin{figure}[h]
\begin{center}
\includegraphics[width=0.88\textwidth]{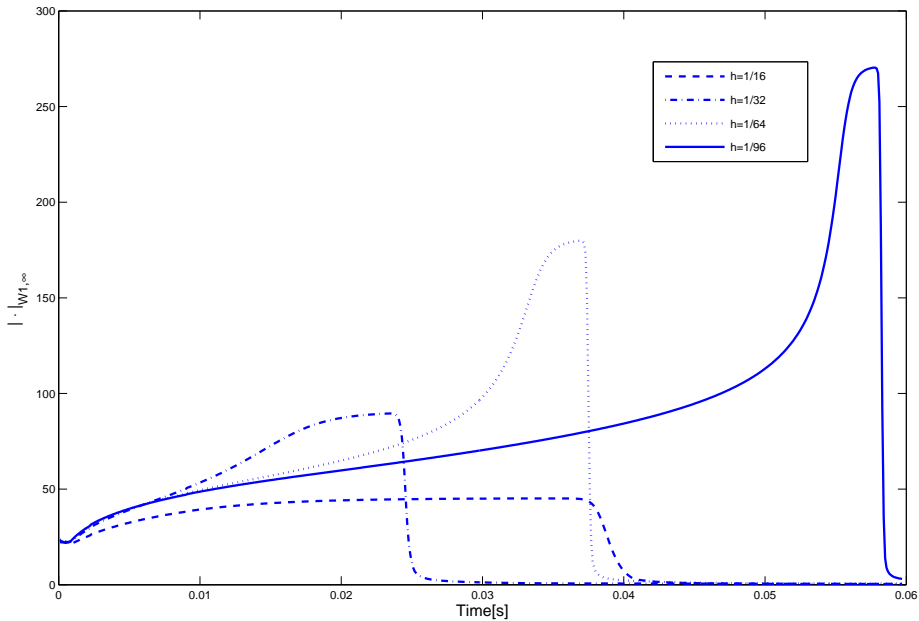}
\end{center}
\caption{ $W^{1, \infty}(\Omega)$ semi-norm from Algorithm \ref{alg} for different values of $h$, with $C^e=C^m=0$ r=5, s=4, $k_a=10^-5$. \label{energy2} }
\end{figure}

In a first experiment, we consider the exchange-only case of LLG and thus neglect magnetostrictive effects as well as all other field contributions, i.e.~$\operator(\cdot) = 0$. We compare the two algorithms for $\alpha=1, s=4, \theta=1, k_a=10^{-5}$, $C^e=C^m=0$, and $T=0.3[s]$. As initital value for $k_m$, we choose $10^{-5}$. In Table~\ref{blowup}, we investigate the overall computation time as well as the empirical blow-up time, which is the time when the $| \textbf{m}(t) |_{1, \infty} = \| \nabla \textbf{m}(t) \|_{L^{\infty}(\Omega) }$ reaches its maximum. We observe that Algorithm~\ref{alg} leads to significantly lower computation time than the midpoint scheme. Moreover, the blow-up time seems to vary between the different schemes. We observed that the blow-up times can be brought more in line with each other if the time-step size $k_a$ is a little decreased (not displayed). As expected, in comparison to the midpoint scheme, Algorithm~\ref{alg} can work with larger time-steps for reasonably small $h$.

\begin{figure}[h]
\begin{center}
\includegraphics[width=0.88\textwidth]{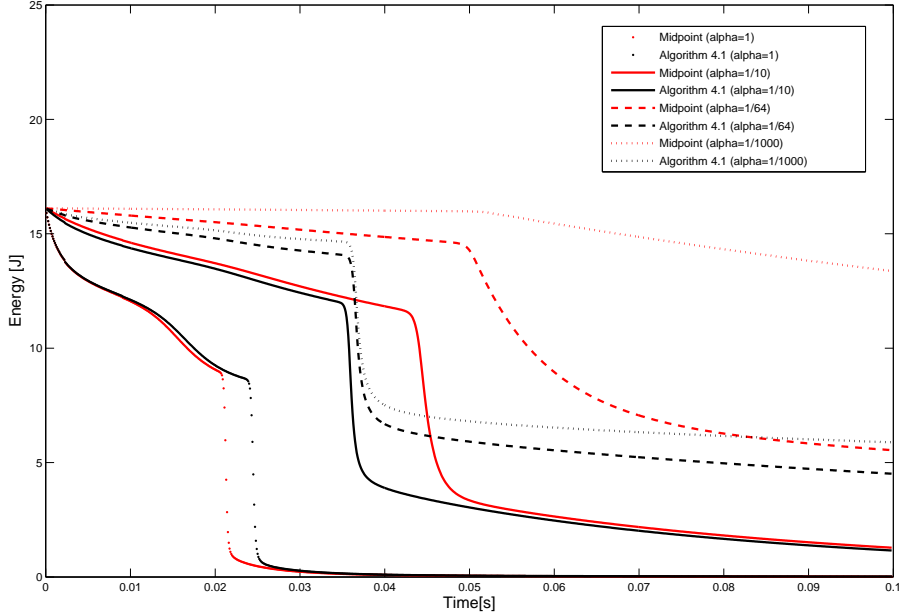}
\end{center}
\caption{Evolution of the Energy for different values of $\alpha$, with $C^e=C^m=0$ r=5, s=4, $k_a=10^-5$, $k_m=10^-5$. \label{energy} }
\end{figure}

On the other hand, the blow-up time seems to increase as $h$ becomes smaller as shown in Figure \ref{energy2}. This effect was not observable in \cite{bp} due to the fixed-point iteration and thus the coupling of 
$h$ and $k_m$. Our empirical observation raises the question of the mere existence and behaviour of the blow-up time in case of exchange only. Put explicitly, the finite time blow-up might be a numerical artifact stemming from insufficient spatial resolution. 

In Figure~\ref{energy}, we plot the discrete energy for the exchange-only case with $C^e=C^m=0$ defined by
$
 E(\textbf{m},t)=\frac{1}{2}\| \nabla\textbf{m}(t)\|_{L^2(\Omega)}^2,
$
and investigate the stability of the respective algorithms
 when $\alpha$ becomes small. The kinks in the graph coincide with the empirical blow-up times.
 We observe that both algorithms seem to be stable, while Algorithm~\ref{alg} provides a stronger energy decay for small values of $\alpha$. 
\begin{table}[h]
\begin{center}
   \begin{tabular}{| c | c | c | c | c | c | c | c | c | }
     \hline
    \multicolumn{1}{|c|}{$\displaystyle$ $1/h$}   & \multicolumn{2}{|c|}{16} & \multicolumn{2}{|c|}{32} & \multicolumn{2}{|c|}{64} & \multicolumn{2}{|c|}{96} \\
     \hline
       & CPU & $T_B$ & CPU & $T_B$  & CPU & $T_B$ & CPU & $T_B$ \\ \hline $\displaystyle$
   Alg \ref{alg}   & 12.9 & $>$ 0.015 & 66.4  & 0.0095 &  578.7 &  0.0079 & 2899 & 0.0074 \\
    Mid&  32.7 & $>$0.015  &  456.7 &  0.0089 &  10415 & 0.0071 & N.A.N & N.A.N \\
     \hline
   \end{tabular}
  \text{ } 
   \caption{ Comparison of CPU and blow-up time $T_B$ for Algorithm \ref{alg} and midpoint scheme \cite{rochat} with $\alpha=1$, T=0.015[s], $ C^e=40$ and $C^m=10.$ \label{blowup2}}
   \end{center}
\end{table}

 \begin{figure}[t]
\begin{center}
\includegraphics[width=.88\textwidth]{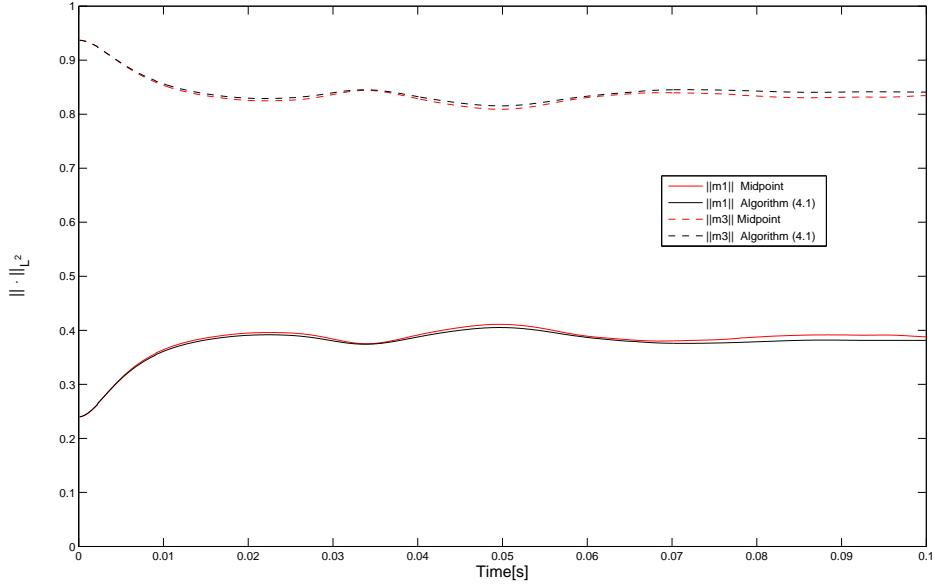}
\end{center}
\caption{Evolution of $\overline{\|m_j\|}_{L^2}, j=1, 3$ with $C^e=C^m=0$ $\alpha=1/64$, r=5, s=4, $k_a=k_m=10^-5$. \label{diff1}}
\end{figure}

\begin{figure}[h]
\begin{center}
\includegraphics[width=.88\textwidth]{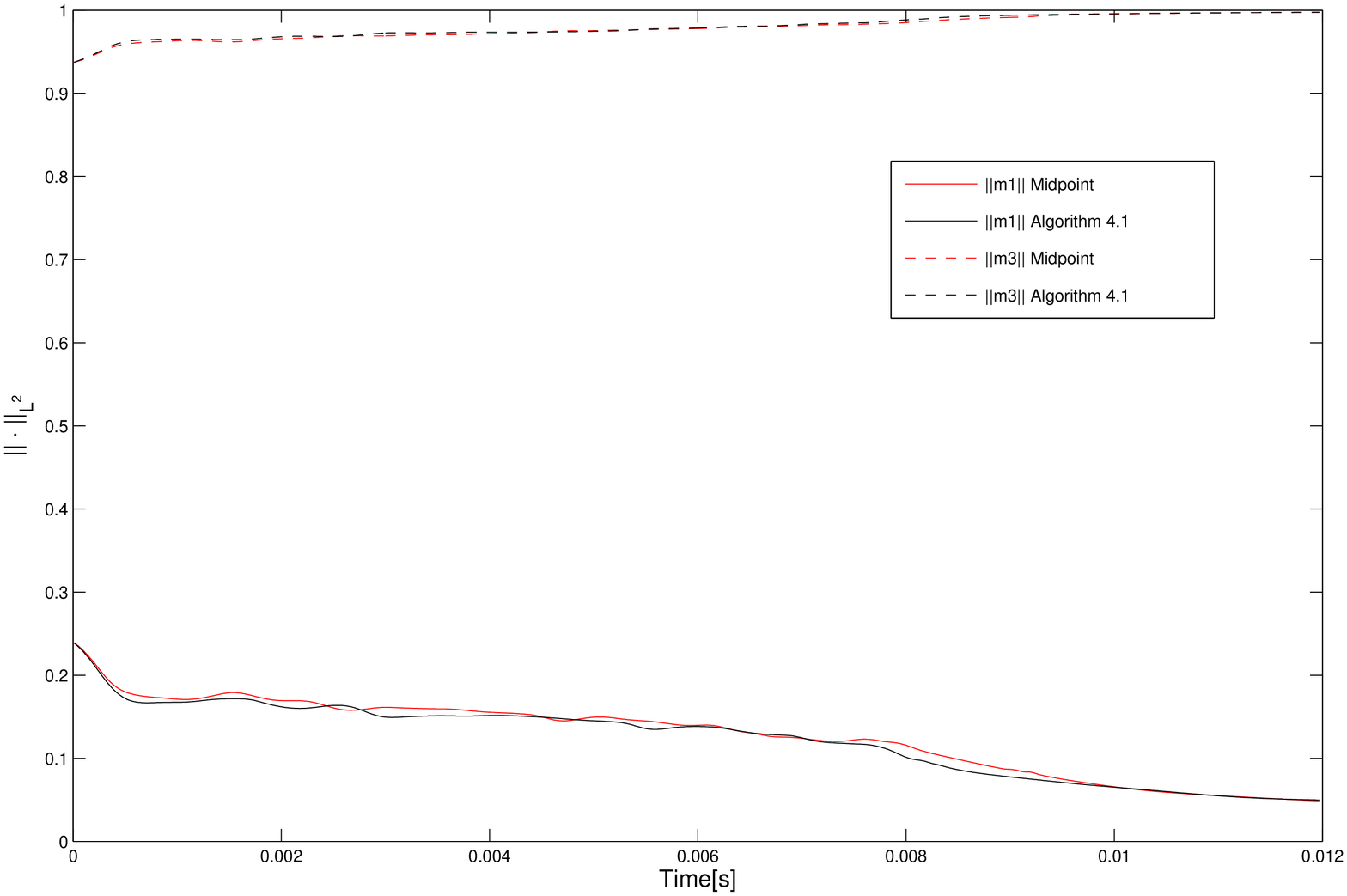}
\end{center}
\caption{Evolution of $\overline{\|m_j\|}_{L^2}, j=1, 3$ with $C^e=40,C^m=10$, $\alpha=1/4$, r=5, s=1 and $k_a=k_m=10^-6$. \label{diff2}} 
\end{figure}
In a second experiment, we include magnetostriction. In Table~\ref{blowup2}, we compare the two algorithms for $C^e = 40, C^m = 10$, and $T = 0.015[s]$ and observe similar results as for the exchange-only case for LLG.  Even with included magnetostrictive effects, the blow-up time varies a lot depending on the spatial resolution $h$, as well as between the different schemes. Finally, in the Figures~\ref{diff1} and~\ref{diff2}, we compare the computational results of the two different schemes. Due to the symmetry of the problem, we only show the evolution of the $L^2(\Omega)$-average
\[
\overline{\| m_j \|}_{L^2(\Omega)} = \frac{1}{| \Omega | } \Big( \int_{\Omega} m_j^2\Big)^{1/2} = \| m_j \|_{L^2(\Omega)}, \quad j=1, 3
\]
 for the $m_1$ and $m_3$ components of the magnetization.


%
%

Overall, we conclude that the results of our algorithm are in good agreement with the midpoint scheme, more feasible for small $\alpha$, throughout much faster to compute, and finally easier to implement. 

\def\new#1{#1}

\noindent \textbf{Acknowledgements.}
Marcus Page and Dirk Praetorius acknowledge financial support through the WWTF project MA09-029 and the FWF project P21732.

\newcommand{\bibentry}[2][!]{\ifthenelse{\equal{#1}{!}}{\bibitem{#2}}{\bibitem[#1]{#2}}}


\begin{thebibliography}{CKNS}

\bibentry{alouges2008}
\textsc{F.~Alouges}:
\emph{A new finite element scheme for Landau-Lifshitz equations}
Discrete and Continuous Dyn. Systems Series S $\mathbf{1}$, 187--196 2008.

\bibentry{alouges2011}
\textsc{F.~Alouges, E.~Kritsikis, J.~Toussaint}:
\emph{A convergent finite element approximation for Landau-Lifshitz-Gilbert equation},
Physica B: Phys.\ Condens.\ Matter, 1--5, 2011.

\bibitem{as}
\textsc{F.~Alouges, A.~Soyeur}:
\emph{On global weak solutions for Landau-Lifshitz equations: existence and nonuniqueness},
Nonlinear Anal.\ \textbf{18}, 1071--1084, 1992.

\bibentry{banas}
\textsc{L'.~Ba\v{n}as, S.~Bartels, A.~Prohl}:
\emph{A convergent implicit finite element discretization of the Maxwell-Landau-Lifshitz-Gilbert equation},
SIAM J. Numer. Anal.~$\mathbf{46}$, 1399--1422, 2008.

\bibentry{maxwell}
\textsc{L'.~Ba\v{n}as, M.~Page, D.~Praetorius}:
\emph{A convergent linear finite-element scheme for the Maxwell-Landau-Lifshitz-Gilbert equation},
ASC Report, Inst.~Anal.~Sci.~Comp., Vienna University of Technology, available through {\tt 	arXiv:1303.4009}, 2013

\bibentry{banas06}
\textsc{L'.~Ba\v{n}as, M.~Slodicka}:
\emph{Error estimates for Landau-Lifshitz-Gilbert equation with magnetostriction},
Appl. Numer. Math. $\mathbf{56}$, 1019--1039, 2006.

\bibentry{bartels}
\textsc{S.~Bartels}:
\emph{Stability and convergence of finite-element approximation schemes for harmonic maps},
SIAM J.\ Numer.\ Anal.\ \textbf{43}, 220--238, 2005.

\bibitem{bjp}
\textsc{S.~Bartels, J.~Ko, A.~Prohl}:
\emph{Numerical analysis of an explicit approximation scheme for the 
Landau-Lifshitz-Gilbert equation},
Math. Comp.\ \textbf{77}, 773--788, 2008.

\bibitem{bp}
\textsc{S.~Bartels, A.~Prohl}:
\emph{Convergence of an implicit finite element method for the 
Landau-Lifshitz-Gilbert equation},
SIAM J. Numer. Anal.\ \textbf{44}, 1405--1419, 2006.

\bibitem{brennerscott}
\textsc{S.~C.~Brenner, L.~R.~Scott}:
\textit{The Mathematical Theory of Finite Element Methods},
Corr. 2nd printing, 2002, Springer, New York, 2002.

\bibentry{multiscale}
\textsc{F.~Bruckner, D.~Suess, M.~Feischl, T.~F\"uhrer, P.~Goldenits, M.~Page, D.~Praetorius}:
\emph{Multiscale modeling in micromagnetics: Well-posedness and numerical integration},
arXiv: 1209.5548, 2012.

\bibentry{carbou}
\textsc{G.~Carbou, M.A.~Efendiev, P.~Fabrie}:
\emph{Global weak solutions for the Landau-Lifschitz equation with magnetostriction},
Math. Meth. Appl. Sci.~$\mathbf{34}$, 1274--1288, 2011.

\bibitem{cimrak}
\textsc{I.~Cimrak}:
\emph{A survey on the numerics and computations for the Landau-Lifshitz 
equation of micromagnetism}, Arch. Comput. Methods Eng.\ \textbf{15}, 277--309, 2008.

\bibentry{elstrodt}
\textsc{J.~Elstrodt}:
\emph{Ma\ss - und Integrationstheorie} (in German),
Springer Verlag, Heidelberg, 6.~Auflage, 2009.

\bibitem{gc}
\textsc{C.J.~Garc\'ia-Cervera}
\emph{Numerical micromagnetics: a review},
Bol. Soc. Esp. Mat. Apl. SeMA \textbf{39}, 103--135, 2007.

\bibitem{tran}
\textsc{K.~N. Le and T.~Tran}
\emph{A convergent finite element approximation for the quasi-static Maxwell--Landau--Lifshitz--Gilbert equations},
arXiv:1212.3369, 1--20, 2012.

\bibitem{mathmod2012}
\textsc{P.~Goldenits, G.~Hrkac, M.~Mayr, D.~Praetorius, D.~Suess}:
\emph{An effective integrator for the Landau-Lifshitz-Gilbert equation},
Proc.~of Mathmod 2012 Conf., 2012.

\bibentry{petra}
\textsc{P.~Goldenits}:
\emph{A convergent geometric time integrator to the Landau-Lifshitz-Gilbert equation} (in German),
Dissertation, Institute of Analysis and Scientific Computing, Vienna University of Technology, 2012

\bibentry{gps}
\textsc{P.~Goldenits, D.~Praetorius, D.~Suess}:
\emph{Convergent geometric integrator for the Landau-Lifshitz-Gilbert equation in micromagnetics},
PAMM: Proc.\ Appl.\ Math.\ Mech. $\mathbf{11}$, 775--776, 2011.

\bibitem{hubertschaefer}
\textsc{A.~Hubert, R.~Sch\"afer}:
\textit{Magnetic Domains. The Analysis of Magnetic Microstructures},
Corr. 3rd printing, 1998, Springer, Heidelberg, 1998.

\bibentry{monk}
\textsc{P.~B.~Monk}:
\emph{Finite Element Methods for Maxwell's Equations},
Oxford University Press, Oxford, UK, 2003.

\bibitem{mp06}
\textsc{M.~Kruzik, A.~Prohl}:
\emph{Recent developments in the modeling, analysis, and numerics of 
ferromagnetism},
SIAM Rev. \textbf{48}, 439--483, 2006.

\bibitem{prohl}
\textsc{A.~Prohl}:
\textit{Computational micromagnetism},
Advances in Numerical Mathematics. 
B.\ G.\ Teubner, Stuttgart, 2001.

\bibentry{diss}
\textsc{M.~Page}:
\emph{On dynamical micromagnetism}, PhD thesis (in progress), Institute of Analysis and Scientific Computing, Vienna University of Technology, 2013.

\bibentry{rochat}
\textsc{J.~Rochat}:
\emph{An implicit finite element method for the Landau-Lifshitz-Gilbert equation with exchange and magnetostriction},
Master's thesis, \'Ecole Polytechnique F\'ed\'erale de Lausanne, 2012.

\bibentry{visintin}
\textsc{A.~Visintin}:
\emph{On Landau-Lifshitz' equations for ferromagnetism},
Japan J. Appl. Math. $\mathbf{2}$, 69--84, 1985.

\end{thebibliography}
\end{document}